\pdfoutput=1
\RequirePackage{ifpdf}
\ifpdf 
\documentclass[pdftex]{sigma}
\else
\documentclass{sigma}
\fi

\numberwithin{equation}{section}

\newtheorem{Theorem}{Theorem}[section]
\newtheorem{Proposition}[Theorem]{Proposition}
 { \theoremstyle{definition}
\newtheorem{Remark}[Theorem]{Remark} }

\begin{document}

\allowdisplaybreaks

\newcommand{\arXivNumber}{1506.00444}

\renewcommand{\PaperNumber}{019}

\FirstPageHeading

\ShortArticleName{The Third, Fifth and Sixth Painlev\'{e} Equations on Weighted Projective Spaces}

\ArticleName{The Third, Fifth and Sixth Painlev\'{e} Equations\\ on Weighted Projective Spaces}

\Author{Hayato CHIBA}

\AuthorNameForHeading{H.~Chiba}

\Address{Institute of Mathematics for Industry, Kyushu University, Fukuoka, 819-0395, Japan}
\Email{\href{mailto:chiba@imi.kyushu-u.ac.jp}{chiba@imi.kyushu-u.ac.jp}}

\ArticleDates{Received September 17, 2015, in f\/inal form February 18, 2016; Published online February 23, 2016}	

\Abstract{The third, f\/ifth and sixth Painlev\'{e} equations are studied by means of
the weighted projective spaces $\mathbb C P^3(p,q,r,s)$ with suitable weights $(p,q,r,s)$
determined by the Newton polyhedrons of the equations.
Singular normal forms of the equations, symplectic atlases of the spaces of
initial conditions, Riccati solutions and Boutroux's coordinates are systematically studied
in a unif\/ied way with the aid of the orbifold structure of $\mathbb C P^3(p,q,r,s)$ and dynamical systems theory.}


\Classification{34M35; 34M45; 34M55}

\section{Introduction}

The f\/irst to sixth Painlev\'{e} equations written in Hamiltonian forms are given by
\begin{gather*}
(\text{P}_{J})\colon \  \frac{dx}{dz} = - \frac{\partial H_J}{\partial y}, \qquad
\frac{dy}{dz} = \frac{\partial H_J}{\partial x},
\end{gather*}
${J} = \text{I}, \text{II}, \text{IV}, \text{III}(D_8), \text{III}(D_7), \text{III}(D_6), \text{V}, \text{VI}$,
with the Hamiltonian functions def\/ined as
\begin{gather*}
H_{\text{I}}  =  \frac{1}{2}x^2 - 2y^3 - zy, \\
H_\text{II}  =  \frac{1}{2}x^2 - \frac{1}{2}y^4 - \frac{1}{2}zy^2 - \alpha y, \\
H_\text{IV}  =  -xy^2 + x^2y - 2xyz - 2\alpha x + 2\beta y, \\
zH_{\text{III} (D_8)}  =  x^2 y^2 - \frac{z}{2y} - \frac{1}{2}y, \\
zH_{\text{III} (D_7)}  =  x^2 y^2 + zx + y + \alpha xy, \\
zH_{\text{III} (D_6)}  =  x^2 y^2- xy^2 + zx + (\alpha + \beta ) xy - \alpha y, \\
zH_\text{V}  =  x(x+z)y (y-1) + \alpha _2 yz - \alpha _3 xy - \alpha _1 x(y-1), \\
z(z-1)H_\text{VI}  =  y(y-1)(y-z)x^2 + \alpha _2 (\alpha _1 + \alpha _2) (y-z) \\
\hphantom{z(z-1)H_\text{VI}  =}{} - \left( \alpha _4 (y-1)(y-z) + \alpha _3 y(y-z) + \alpha _0 y (y-1) \right)x,
\end{gather*}
where $\alpha_i, \beta \in \mathbb C$ are arbitrary parameters.
Parameters of~$H_\text{VI}$ satisfy the constraint $\alpha _0 + \alpha _1 + 2 \alpha _2 + \alpha _3 + \alpha _4 = 0$.
The third Painlev\'{e} equation is divided into three cases $\text{III} (D_8)$, $\text{III} (D_7)$, $\text{III} (D_6)$
due to the geometry of the spaces of initial conditions \cite{Ohy, Sak}.
See~\cite{Tsu} for the list of Hamiltonians and B\"{a}cklund transformations written in these coordinates,
although~$H_{\text{III} (D_8)}$ here is obtained by putting $x \mapsto x - 1/(2y)$ for that in~\cite{Tsu}.

Among these Hamiltonians, only $(\text{P}_\text{I})$, $(\text{P}_\text{II})$ and $(\text{P}_\text{IV})$
are polynomials with respect to both of the dependent variables $x$, $y$ and the independent variable $z$.
In~\cite{Chi2, Chi1}, $(\text{P}_\text{I})$,~$(\text{P}_\text{II})$ and~$(\text{P}_\text{IV})$
are studied by means of a weighted projective space $\mathbb C P^{3}(p,q,r,s)$,
whose weight $(p,q,r,s)$ is one of the invariants of the equation
determined by the Newton polyhedron.
In particular, the Painlev\'{e} property, the spaces of initial conditions and Kovalevskaya exponents
are investigated in detail.
The purpose in this paper is to extend the previous result to the third, f\/ifth and sixth Painlev\'{e}
equations, whose Newton polyhedrons are degenerate.

According to \cite{Chi2, Chi1}, let us recall the def\/inition of the Newton diagram of a polynomial
dif\/ferential system.
Consider the system of polynomial dif\/ferential equations
\begin{gather}
\frac{dx_i}{dz} = f_i(x_1, \dots ,x_m, z), \qquad i=1, \dots ,m.
\label{1-2}
\end{gather}
The exponent of a monomial $x_1^{\mu_1} \cdots x_m^{\mu_m}z^\eta$
included in the right hand side $f_i$ is def\/ined as
$(\mu_1, \dots , \mu_{i-1}, \mu_{i}-1 , \mu_{i+1}, \dots , \mu_{m}, \eta +1)$.
Each exponent specif\/ies a point of the integer lattice in $\mathbb R^{m+1}$.
The Newton polyhedron of (\ref{1-2}) is the convex hull of the union of the positive quadrants
$\mathbb R_+^{m+1}$ with vertices at the exponents of the monomials which appear in the system.
The Newton diagram of the system is the union of the compact faces of its Newton polyhedron.

We also consider the perturbative system
\begin{gather}
\frac{dx_i}{dz} = f_i(x_1 , \dots ,x_m,z) + g_i(x_1 , \dots ,x_m,z), \qquad i=1, \dots ,m,
\label{1-3}
\end{gather}
where $f_i$ and $g_i$ are polynomials.
We suppose that
\begin{enumerate}\itemsep=0pt
\item[(A1)] the Newton polyhedron of the truncated system~(\ref{1-2}) has only one compact face
and all exponents of monomials included in~(\ref{1-2}) lie on the face.
\end{enumerate}

In this case, there is a tuple of relatively prime \textit{positive} integers $(p_1, \dots ,p_m,r,s)$
and a~hyperplane $p_1x_1 + \cdots + p_mx_m + rz = s$ in $\mathbb R^{m+1}$ such that all exponents lie on the plane;
i.e., any monomials $x_1^{\mu_1} \cdots x_m^{\mu_m}z^\eta$ included in $f_i$ satisfy
\begin{gather*}
p_1 \mu_1 + \cdots + p_i(\mu_i-1) + \cdots + p_m \mu_m + r(\eta + 1) = s.
\end{gather*}
In \cite{Chi2}, we further suppose that $s-r = 1$ though it is not so essential.
To regard $g_i$ as a~perturbation, we suppose that
\begin{enumerate}\itemsep=0pt
\item[(A2)] any monomials $x_1^{\mu_1} \cdots x_m^{\mu_m}z^\eta$ included in~$g_i$ satisfy
\begin{gather*}
p_1\mu_1 + \cdots + p_i (\mu_i - 1) +\cdots + p_m\mu_m + r(\eta + 1) < s
\end{gather*}
(this implies that the exponents of $g_i$ lie on the lower side of the hyperplane).
\end{enumerate}

Due to the property of the Newton diagram, it is easy to verify that
the truncated system~(\ref{1-2}) is invariant under the $\mathbb Z_s$-action given by
\begin{gather}
(x_1, \dots ,x_m,z) \mapsto \big(\omega ^{p_1}x_1, \dots ,\omega ^{p_m}x_m, \omega ^{r}z\big), \qquad \omega := e^{2\pi i/s}.
\label{1-6}
\end{gather}
We require the same for equation~(\ref{1-3}):
\begin{enumerate}\itemsep=0pt
\item[(A3)]
Equation~(\ref{1-3}) is invariant under the $\mathbb Z_s$-action (\ref{1-6}).
\end{enumerate}

A tuple of positive integers $(p_1, \dots , p_m, r, s)$ is called the weight of the system~(\ref{1-3}).
It is known that there is a one-to-one correspondence between nondegenerate Newton diagrams and toric varieties.
If exponents lie on the unique plane $p_1 x_1 + \cdots + p_m x_m + rx = s$ (assumption~(A1)),
then the associated toric variety is the weighted projective space $\mathbb C P^{m+1} (p_1, \dots ,p_m r, s)$,
which is an~$(m+1)$-dimensional orbifold, see Section~\ref{section2.1} for the def\/inition.

The f\/irst, second and fourth Painlev\'{e} equations satisfy the above assumptions.
For the f\/irst Painlev\'{e} equation $x' = 6y^2 + z$, $y' = x$, put $f = (6y^2 + z, x)$ and $g= (0,0)$.
The Newton diagram is determined by three points $(-1,2,1)$, $(-1,0,2)$ and $(1,-1,1)$.
They lie on the unique plane $3x+2y + 4z = 5$.
For the second Painlev\'{e} equation $x' = 2y^3 + yz + \alpha $, $y' = x$ with a~parameter~$\alpha $,
put $f = (2y^3 + yz, x)$ and $g= (\alpha ,0)$.
The Newton diagram is determined by points $(-1,3,1)$, $(-1,1,2)$, $(1,-1,1)$, which lie on the plane $2x + y+ 2z = 3$.
For the fourth Painlev\'{e} equation, put $f = (-x^2 + 2xy + 2xz, -y^2+2xy-2yz)$ and $g= (- 2\beta ,-2\alpha )$.
The Newton diagram is given by the unique face on the plane $x + y+ z = 2$
passing through the exponents $(1,0,1)$, $(0,1,1)$ and $(0,0,2)$.
Hence, the weighted projective spaces associated with them are given by
$\mathbb C P^3(3,2,4,5)$, $\mathbb C P^3(2,1,2,3)$ and $\mathbb C P^3(1,1,1,2)$, respectively.
Note that~$g$ consists of terms including arbitrary parameters~$\alpha$,~$\beta$ for these systems.

The orbifold $\mathbb C P^{m+1}(p_1, \dots , p_m,r,s)$ is regarded as a compactif\/ication of the phase space
$\mathbb C^{m+1} = \{ (x_1, \dots ,x_m, z)\}$ of the system~(\ref{1-3}).
In~\cite{Chi2, Chi1}, the system (\ref{1-3}), in particular the f\/irst, second and fourth Painlev\'{e}
equations, are studied with the aid of the geometry of $\mathbb C P^{m+1}(p_1, \dots$, $p_m,r,s)$.
In this paper, the third, f\/ifth and sixth Painlev\'{e} equations will be investigated,
for which the Newton polyhedrons are degenerate and do not satisfy (A1).

The third Painlev\'{e} equation of type $D_6$ is explicitly given by
\begin{gather}
(\text{P}_{\text{III} (D_6)})\colon \   \begin{cases}
zx' = -2 x^2y + 2 xy - (\alpha + \beta) x + \alpha, \\
zy' = 2 xy^2 - y^2 +z + (\alpha + \beta ) y .
\end{cases}
\label{1-7}
\end{gather}
Put $f = (-2 x^2y + 2 xy, 2 xy^2 - y^2 + z)$
and $g = (- (\alpha +\beta) x + \alpha, (\alpha +\beta) y)$.
Exponents of $f$ are given by $(1,1,0)$, $(0,1,0)$, $(0,-1,1)$.
The Newton polyhedron generated by them has no compact faces because
the positive quadrant~$\mathbb R_+^{3}$ with a vertex at~$(0,1,0)$ includes~$(1,1,0)$;
the Newton diagram is empty.
Nevertheless, these three points lie on the plane $y + 2 z = 1$.
Hence, we def\/ine the weight of~$(\text{P}_{\text{III} (D_6)})$ by~$(0,1,2,1)$
and use the weighted projective space $\mathbb C P^{3}(0,1,2,1)$, which is \textit{not} compact
because of the nonpositive weight.

The third Painlev\'{e} equation of type $D_7$ is given by
\begin{gather*}
(\text{P}_{\text{III} (D_7)})\colon \   \begin{cases}
zx' = -2 x^2y -1 - \alpha x, \\
zy' = 2 xy^2 +z + \alpha y .
\end{cases}
\end{gather*}
Put $f = (-2 x^2y -1, 2 xy^2 +z)$
and $g = (- \alpha x, \alpha y)$.
Exponents of $f$ are given by $(1,1,0), (-1,0,0)$ and $(0,-1,1)$.
The Newton diagram is again empty, however, these three points lie on the plane $-x + 2y + 3z = 1$.
Thus we def\/ine the weight of $(\text{P}_{\text{III} (D_7)})$ by $(-1,2,3,1)$
and use the weighted projective space $\mathbb C P^{3}(-1,2,3,1)$.

The third Painlev\'{e} equation of type $D_8$ is given by
\begin{gather*}
(\text{P}_{\text{III} (D_8)})\colon \   \begin{cases}
\displaystyle zx' = \frac{1}{2} -2 x^2y - \frac{z}{2y^2}, \\
  zy' = 2 xy^2 .
\end{cases}
\end{gather*}
There are no parameters and put $g = (0,0)$.
Exponents of the system are $(1,1,0)$, $(-1,0,0)$ and $(-1,-2,1)$.
The Newton polyhedron is degenerate, however, these three points lie on the plane $-x + 2y + 4z = 1$.
Thus we def\/ine the weight of $(\text{P}_{\text{III} (D_8)})$ to be~$(-1,2,4,1)$
and use the weighted projective space $\mathbb C P^{3}(-1,2,4,1)$.

The f\/ifth Painlev\'{e} equation is given by
\begin{gather*}
(\text{P}_{\text{V}})\colon \  \begin{cases}
  zx' = -2 x^2y +x^2+xz - 2xyz + (\alpha _1 + \alpha _3) x - \alpha _2 z,  \\
  zy' = 2 xy^2 - 2xy - yz + y^2z - (\alpha _1 + \alpha_3 ) y + \alpha _1.
\end{cases}
\end{gather*}
Put $f = (-2 x^2y +x^2+xz - 2xyz, 2 xy^2 - 2xy - yz + y^2z)$ and
$g= ((\alpha _1 + \alpha _3) x - \alpha _2 z, - (\alpha _1 + \alpha_3 ) y + \alpha _1)$.
Exponents of~$f$ are $(1,1,0)$, $(1,0,0)$, $(0,0,1)$ and $(0,1,1)$.
Although there are four exponents, they lie on the unique plane $x + z = 1$.
Thus we def\/ine the weight of $(\text{P}_{\text{V} })$ by $(1,0,1,1)$
and use the weighted projective space~$\mathbb C P^{3}(1,0,1,1)$.

The sixth Painlev\'{e} equation is given by
\begin{gather*}
(\text{P}_{\text{VI}})\colon \    \begin{cases}
  z(z-1)x' = -3 x^2y^2 +2 x^2y + 2x^2yz - x^2z + g_1 ,  \\
  z(z-1)y' = 2 xy^3 - 2xy^2 - 2 xy^2z + 2 xyz + g_2,
\end{cases}
\end{gather*}
where $(g_1, g_2)$ consists of terms including parameters.
Exponents of the other terms are given by $(1,2,0)$, $(1,1,0)$, $(1,1,1)$ and $(1,0,1)$,
which lie on the unique plane $x = 1$.
Thus we def\/ine the weight of $(\text{P}_{\text{VI} })$ as $(1,0,0,1)$
and use the weighted projective space $\mathbb C P^{3}(1,0,0,1)$.

\begin{table}[h]
\centering
\begin{tabular}{|c||c|c|c|c|c|c|}
\hline
 & $(p,q,r,s)$ & $\deg  (H)$ & $\kappa$ & $\lambda _1$, $\lambda _2$, $\lambda _3$ & $l$ & $c$ \\ \hline \hline
$\text{P}_\text{I}$ & $(3,2,4,5)$ & 6 & 6 & $6,4,5$ & 1 & 2 \\ \hline
$\text{P}_\text{II}$ & $(2,1,2,3)$ & 4 & 4 & $4,2,3$ & 2 & 3 \\ \hline
$\text{P}_\text{IV}$ & $(1,1,1,2)$ & 3 & 3 & $3,1,2$ & 3 & 4 \\ \hline
$\text{P}_{\text{III} (D_8)}$ & $(-1,2,4,1)$ & 2 & 2 & $2,4,1$ & 2 & 3 \\ \hline
$\text{P}_{\text{III} (D_7)}$ & $(-1,2,3,1)$ & 2 & 2 & $2,3,1$ & 2 & 3 \\ \hline
$\text{P}_{\text{III} (D_6)}$ & $(0,1,2,1)$ & 2 & 2 & $2,2,1$ & 3 & 4 \\ \hline
$\text{P}_\text{V}$ & $(1,0,1,1)$ & 2 & 2 & $2,1,1$ & 4 & 5 \\ \hline
$\text{P}_\text{VI}$ & $(1,0,0,1)$ & 2 & 2 & $2,0,1$ & 5 & 6 \\ \hline
\end{tabular}

\caption{$\deg  (H)$ denotes the weighted degree of the Hamiltonian function with
respect to the weight $\deg (x) = p$, $\deg (y) = q$, $\deg (z) = r$.
$\kappa$ denotes the Kovalevskaya exponent def\/ined in Section~\ref{section2.2}.
$(\lambda _1, \lambda _2, \lambda _3)$ is the weight for the weighted blow-up, which
coincides with the eigenvalues of the Jacobi matrix at the singularity.
In~\cite{Chi2}, it is proved that $\deg  (H) = \kappa = \lambda _1$
and $r = \lambda _2$, $s = \lambda _3$.
$l$~gives the number of types of Laurent series solutions given in Section~\ref{section2.2},
and~$c$ is the number of local charts for the space of initial conditions.
We always have~$c = l+1$.}\label{table1}
\end{table}

In this paper, the third, f\/ifth and sixth Painlev\'{e} equations are studied
with the aid of the weighted projective spaces $\mathbb C P^3(p,q,r,s)$ and dynamical systems theory.
The weighted projective space is decomposed into the disjoint sum $\mathbb C P^3(p,q,r,s) \simeq \mathbb C P^2(p,q,r) \cup \mathbb C^3$.
This implies that the natural phase space $\mathbb C^3 = \{ (x,y,z)\}$ of the Painlev\'{e} equations
is embedded in $\mathbb C P^3(p,q,r,s)$ and the two-dimensional manifold $\mathbb C P^2(p,q,r)$ is attached at inf\/inity.
We regard the Painlev\'{e} equation as a~three-dimensional autonomous vector f\/ield def\/ined on $\mathbb C P^3(p,q,r,s)$.
Then, the vector f\/ield has several f\/ixed points on the inf\/inity set $\mathbb C P^2(p,q,r)$.
These f\/ixed points describe the asymptotic behavior of solutions.
Some of these f\/ixed points correspond to movable singularities, and the other correspond to the irregular singular point.

\looseness=-1
The dynamical systems theory is applied to these f\/ixed points to investigate the Painlev\'{e} equations.
The singular normal form of the Painlev\'{e} equation~\cite{Chi1, CosCos},
which is a local integrable system around
a movable singularity, is obtained by applying the normal form theory around f\/ixed points.
The space of initial conditions and its symplectic atlas are constructed by the weighted blow-up
at these f\/ixed points.
The weight for the weighted blow-up, which is also an invariant of the Painlev\'{e} equation related to
the Kovalevskaya exponent \cite{Chi2, Chi1}, is determined by eigenvalues of the Jacobi matrix
of the vector f\/ield at the f\/ixed points.
It is known that the Painlev\'{e} equations are reduced to the Riccati equations
when the parameters take certain specif\/ic values.
Such Riccati solutions are characterized as a center (un)stable manifold at the f\/ixed point on $\mathbb C P^3(p,q,r,s)$.
Although some of these results are well known for experts, our new approach based on the weighted projective space
and dynamical systems theory provides a~systematic way to investigate them.
From our analysis, it turns out that the weights and the Kovalevskaya exponents are important invariants
of the Painlev\'{e} equations.
In particular, the Painlev\'{e} equations may be classif\/ied by these invariants, which will be reported
in a forthcoming paper.

Our method will be explained in detail for the third Painlev\'{e} equation of type $D_6$ in Section~\ref{section3}.
Since the strategy for the other Painlev\'{e} equations $(\text{P}_{\text{III}(D_7)})$,
$(\text{P}_{\text{III}(D_8)})$, $(\text{P}_{\text{V}})$ and $(\text{P}_{\text{VI}})$ is
completely the same as that for $(\text{P}_{\text{III}(D_6)})$, we only show a sketch and several formulae
for them after Section~\ref{section4}.
See \cite{Chi1} for $(\text{P}_{\text{I}})$, $(\text{P}_{\text{II}})$ and $(\text{P}_{\text{IV}})$.

\section{Settings}\label{section2}

\subsection{Weighted projective spaces}\label{section2.1}

For a tuple of integers $(p_1, \dots ,p_m, r, s)$,
consider the weighted $\mathbb C^*$-action on $\mathbb C^{m+2}$ def\/ined by
\begin{gather*}
(x_1, \dots , x_m, z,\varepsilon ) \mapsto
\big(\lambda ^{p_1}x_1 , \dots ,\lambda ^{p_m} x_m, \lambda ^rz, \lambda ^s \varepsilon \big), \qquad
\lambda \in \mathbb C^* := \mathbb C \backslash \{ 0\}.
\end{gather*}
The quotient space
\begin{gather*}
\mathbb C P^{m+1}(p_1, \dots , p_m,r,s) := \mathbb C^{m+2}\backslash \{ 0\} /\mathbb C^*
\end{gather*}
gives an $(m+1)$-dimensional orbifold called the weighted projective space.
In this paper, we only use a three-dimensional space.
When $(p,q,r,s)$ are positive integers, the orbifold structure of $\mathbb C P^3(p,q,r,s)$ is obtained as follows:

The space $\mathbb C P^3(p,q,r,s)$ is def\/ined by the equivalence relation on $\mathbb C^4\backslash \{ 0\}$
\begin{gather*}
(x,y,z, \varepsilon ) \sim \big(\lambda^p x, \lambda ^qy, \lambda ^rz, \lambda ^s \varepsilon \big).
\end{gather*}

(i) When $x\neq 0$,
\begin{gather*}
(x,y,z,\varepsilon ) \sim \big(1,  x^{-q/p}y,   x^{-r/p}z,  x^{-s/p}\varepsilon \big)
=:(1, Y_1, Z_1, \varepsilon _1).
\end{gather*}
Due to the choice of the branch of $x^{1/p}$, we also obtain
\begin{gather*}
(Y_1, Z_1, \varepsilon _1) \sim
\big(e^{-2q\pi i /p}Y_1, e^{-2r\pi i /p}Z_1, e^{-2s\pi i /p}\varepsilon _1\big),
\end{gather*}
by putting $x \mapsto e^{2\pi i }x$.
This implies that the subset of $\mathbb C P^3(p,q,r,s)$ such that $x\neq 0$ is homeomorphic to $\mathbb C^3 / \mathbb Z_p$,
where the $\mathbb Z_p$-action is def\/ined as above.

(ii) When $y\neq 0$,
\begin{gather*}
(x,y,z,\varepsilon ) \sim \big(y^{-p/q}x,  1 ,  y^{-r/q}z,  y^{-s/q}\varepsilon \big)
=:(X_2, 1, Z_2, \varepsilon _2).
\end{gather*}
Because of the choice of the branch of~$y^{1/q}$, we obtain
\begin{gather*}
(X_2, Z_2, \varepsilon _2) \sim \big(e^{-2p\pi i/q}X_2, e^{-2r\pi i/q}Z_2, e^{-2s\pi i/q} \varepsilon _2\big).
\end{gather*}
Hence, the subset of $\mathbb C P^3(p,q,r,s)$ with $y\neq 0$ is homeomorphic to~$\mathbb C^3 / \mathbb Z_q$.

(iii) When $z\neq 0$,
\begin{gather*}
(x,y,z,\varepsilon ) \sim \big(z^{-p/r}x,  z^{-q/r}y ,  1,  z^{-s/r}\varepsilon \big)
=: (X_3, Y_3, 1, \varepsilon _3).
\end{gather*}
Similarly, the subset $\{ z \neq 0\} \subset \mathbb C P^3(p,q,r,s)$ is homeomorphic to~$\mathbb C^3 / \mathbb Z_r$.

(iv) When $\varepsilon \neq 0$,{\samepage
\begin{gather*}
(x,y,z,\varepsilon ) \sim \big(\varepsilon ^{-p/s}x,  \varepsilon ^{-q/s}y ,  \varepsilon ^{-r/s}z ,  1\big)
=:(X_4, Y_4, Z_4, 1).
\end{gather*}
The subset $\{ \varepsilon \neq 0\} \subset \mathbb C P^3(p,q,r,s)$ is homeomorphic to~$\mathbb C^3 / \mathbb Z_s$.}

This proves that the orbifold structure of~$\mathbb C P^3(p,q,r,s)$ is given by
\begin{gather*}
\mathbb C P^3(p,q,r,s) = \mathbb C^3/\mathbb Z_p  \cup   \mathbb C^3/\mathbb Z_q  \cup   \mathbb C^3/\mathbb Z_r   \cup   \mathbb C^3/\mathbb Z_s.
\end{gather*}
The local charts $(Y_1, Z_1, \varepsilon _1)$, $(X_2, Z_2, \varepsilon _2)$,
$(X_3, Y_3, \varepsilon _3)$ and $(X_4, Y_4, Z_4)$ def\/ined above are called inhomogeneous coordinates
as the usual projective space.
Note that they give coordinates on the lift~$\mathbb C^3$, not on the quotient $\mathbb C^3 / \mathbb Z_i$ $(i=p,q,r,s)$.
Therefore, any equations written in these inhomogeneous coordinates should be invariant under
the corresponding~$\mathbb Z_i$ actions.

In what follows, we use the notation $(x,y,z)$ for the fourth local chart instead of~$(X_4, Y_4, Z_4)$
because the Painlev\'{e} equation will be given on this chart.
The transformations between inhomogeneous coordinates are give by
\begin{gather}
x =   \varepsilon _1^{-p/s} =   X_2\varepsilon _2^{-p/s} =  X_3\varepsilon _3^{-p/s}, \qquad
y =   Y_1\varepsilon _1^{-q/s} =   \varepsilon _2^{-q/s}=  Y_3\varepsilon _3^{-q/s},\nonumber\\
z =   Z_1\varepsilon_1 ^{-r/s} =   Z_2\varepsilon _2^{-r/s}=  \varepsilon _3^{-r/s}.\label{2-3}
\end{gather}
The same transformation rule holds even if $p$, $q$, $r$, $s$ include negative integers.
If there are $0$ among them, for example if $p=0$, then we have
$\mathbb C P^3(0,q,r,s) \simeq \mathbb C \times \mathbb C P^2(q,r,s)$.
$\mathbb C P^3(p,q,r,s)$ is compact if and only if all $p$, $q$, $r$, $s$ are positive.

\subsection{Laurent series solutions and Kovalevskaya exponents}\label{section2.2}

To construct the space of initial conditions, we need the expressions of the Laurent series of solutions.
Let $(p,q,r,s)$ be the weight of a given system determined by the Newton polyhedron.
Suppose that the system has a Laurent series solution of the form
\begin{gather*}
 x = \sum^\infty_{n=0} A_n (z-z_0)^{-p+n}, \qquad
 y = \sum^\infty_{n=0} B_n (z-z_0)^{-q+n},
\end{gather*}
where $(A_0, B_0) \neq (0,0)$ and $z_0$ is a movable pole.
Such a Laurent series solution is called \textit{regular}.
A Laurent series solution is called \textit{exceptional} if it is not expressed in this form;
i.e., $(A_0, B_0) = (0,0)$ or the order of a pole of either $x$ or $y$ is larger than~$p$ or $q$, respectively.
If a regular Laurent series represents a general solution of the system,
it includes an arbitrary parameter, which depends on initial conditions, other than~$z_0$.
The smallest integer~$\kappa$ such that $(A_\kappa, B_\kappa)$ includes an arbitrary parameter is called
the Kovalevskaya exponent.
In~\cite{Chi2}, it is proved that the Kovalevskaya exponent
of the regular Laurent series solution is invariant under a certain
class of coordinates transformations including the automorphism group of $\mathbb C P^3(p,q,r,s)$.
For the f\/irst, second and fourth Painlev\'{e} equations, all Laurent series solutions are regular
because they satisfy the assumptions~(A1) and~(A2).
The third, f\/ifth and sixth Painlev\'{e} equations have exceptional Laurent series solutions,
however, they can be converted into the regular series by the B\"{a}cklund transformations.
Hence, the Kovalevskaya exponents $\kappa$ of exceptional Laurent series solutions are well-def\/ined
and given as in Table~\ref{table1}.
In what follows, denote $T := z-z_0$.

($\text{P}_{\text{III} (D_6)}$):
The third Painlev\'{e} equation of type $D_6$ has three types of Laurent series solutions given by
\begin{alignat*}{3}
& \mathrm{(i)} \quad & &
 x = 1 + \frac{\beta}{z_0}\cdot T + A_2\cdot T^2 + O\big(T^3\big), &\\
&&&  y = -z_0\cdot T^{-1} + \frac{1}{2}(-1-\alpha +\beta) + B_2 \cdot T + O\big(T^2\big),& \\
& \mathrm{(ii)} \quad & &
x = 0 \cdot T^0 - \frac{\alpha}{z_0}\cdot T + A_2\cdot T^2 + O\big(T^3\big), & \\
&&& y = z_0\cdot T^{-1} + \frac{1}{2}(1-\alpha + \beta) + B_2 \cdot T + O\big(T^2\big),&  \\
& \mathrm{(iii)} \quad & &
  x = -z_0 \cdot T^{-2} + 0\cdot T^{-1} + A_2 + O(T),& \\
&&&  y = -T + \frac{-2+\alpha + \beta}{2z_0} \cdot T^2 + B_2 \cdot T^3 + O\big(T^4\big), &
\end{alignat*}
where $A_2$ is an arbitrary constant and $B_2$ is a certain function of $A_2$.
Since $(p,q) = (0,1)$, the f\/irst two series are regular, while the last one is exceptional.
The Kovalevskaya exponents of all series are $\kappa = 2$.

($\text{P}_{\text{III} (D_7)}$):
The third Painlev\'{e} equation of type $D_7$ has two types of Laurent series solutions given by
\begin{alignat*}{3}
& \mathrm{(i)}\quad & &
  x = \frac{1}{z_0} \cdot T + \frac{\alpha -1}{2 z_0^2}\cdot T^2 + A_2\cdot T^3 + O\big(T^4\big), & \\
&&&  y = -z_0^2 \cdot T^{-2} -z_0 \cdot T^{-1} + B_2 + O(T), & \\
& \mathrm{(ii)} & & x = -z_0 \cdot T^{-2} + 0 \cdot T^{-1} + A_2 + O(T), & \\
&&&  y = -T^{1} + \frac{\alpha -2}{2 z_0} \cdot T^2 + B_2 \cdot T^3 + O\big(T^4\big),  &
\end{alignat*}
where $A_2$ is an arbitrary constant and~$B_2$ is a certain function of~$A_2$.
Since $(p,q) = (-1,2)$, the former series is regular, while the latter one is exceptional.
The Kovalevskaya exponents of both series are $\kappa = 2$.

($\text{P}_{\text{III} (D_8)}$):
The third Painlev\'{e} equation of type $D_8$ has two types of Laurent series solutions given by
\begin{alignat*}{3}
& \mathrm{(i)} \quad & &
  x = -\frac{1}{2 z_0} \cdot T + \frac{1}{4 z_0^2}\cdot T^2 + \frac{2B_2-1}{4z_0^2} \cdot T^3 + O\big(T^4\big),& \\
&&& y = 2z_0^2 \cdot T^{-2} +2z_0 \cdot T^{-1} + B_2 + O(T), & \\
& \mathrm{(ii)} \quad & &
  x = 2z_0^2 \cdot T^{-3} + 2z_0 \cdot T^{-2} + 0 \cdot T^{-1} + O(1), & \\
&&&  y = \frac{1}{2z_0}T^{2} + 0\cdot T^3 + B_2 \cdot T^4 + O\big(T^5\big), &
\end{alignat*}
where $B_2$ is an arbitrary constant.
Since $(p,q) = (-1,2)$, the former series is regular, while the latter one is exceptional.
The Kovalevskaya exponents of both series are $\kappa = 2$.

($\text{P}_{\text{V}}$):
The f\/ifth Painlev\'{e} equation has four types of Laurent series solutions given by
\begin{alignat*}{3}
& \mathrm{(i)} \quad & &
 x = z_0 \cdot T^{-1} + \frac{1}{2}(1-z_0+\alpha_1-\alpha _3) + A_2\cdot T + O\big(T^2\big),&  \\
&&& y = 1 + \frac{\alpha _3}{z_0} \cdot T + B_2 \cdot T^2 + O\big(T^3\big), & \\
& \mathrm{(ii)} \quad & &
 x = -z_0 \cdot T^{-1} + \frac{1}{2}(-1-z_0+\alpha _1-\alpha _3) + A_2\cdot T + O\big(T^2\big),& \\
&&&  y = 0\cdot T^0 - \frac{\alpha _1}{z_0} \cdot T + B_2 \cdot T^2 + O\big(T^3\big), & \\
& \mathrm{(iii)} \quad & &
 x = -z_0 + (\alpha _1 + \alpha _2 + \alpha _3-2) \cdot T + A_2\cdot T^2 + O\big(T^3\big), & \\
&&& y = T^{-1} + \frac{z_0 + \alpha _1+2 \alpha _2 + \alpha _3-2}{2z_0} + B_2 \cdot T + O\big(T^2\big),& \\
& \mathrm{(iv)} \quad & &
 x = 0 \cdot T^0 + \alpha _2 \cdot T + A_2\cdot T^2 + O\big(T^3\big), & \\
&&& y = -T^{-1} + \frac{z_0 + \alpha _1+2 \alpha _2 + \alpha _3}{2z_0} + B_2 \cdot T + O\big(T^2\big), &
\end{alignat*}
where $A_2$ is an arbitrary constant and $B_2$ is a certain function of $A_2$.
Since $(p,q) = (1,0)$, (i)~and~(ii) are regular, while~(iii) and~(iv) are exceptional.
The Kovalevskaya exponents of all series are $\kappa = 2$.

($\text{P}_{\text{VI}}$):
The sixth Painlev\'{e} equation has f\/ive types of Laurent series solutions given by
\begin{alignat*}{3}
& \mathrm{(i)} \quad & &
 x = T^{-1} +
 \frac{2 - 4z_0 + \alpha _0 - 2 z_0 \alpha _0 + z_0 \alpha _3 - \alpha _4 + z_0 \alpha _4}{2z_0(z_0-1)}
+ A_2\cdot T + O\big(T^2\big), & \\
&&& y = z_0 + (2+\alpha _0) \cdot T + B_2 \cdot T^2 + O\big(T^3\big), & \\
& \mathrm{(ii)} \quad & &   x = -z_0 \cdot T^{-1}
 + \frac{1 - z_0- \alpha _0 + 2 \alpha _3 - z_0 \alpha _3 - \alpha _4 + z_0 \alpha _4}{2(z_0-1)}
+ A_2\cdot T + O\big(T^2\big), & \\
&&& y = 1 -\frac{\alpha _3}{z_0} \cdot T + B_2 \cdot T^2 + O\big(T^3\big),& \\
& \mathrm{(iii)} \quad & &
  x = (z_0-1) \cdot T^{-1} + \frac{z_0 - \alpha _0 - z_0 \alpha _3 + \alpha _4 + z_0 \alpha _4}{2z_0}
+ A_2\cdot T + O\big(T^2\big),& \\
&&& y = 0 \cdot T^0 + \frac{\alpha _4}{z_0-1} \cdot T + B_2 \cdot T^2 + O\big(T^3\big), & \\
& \mathrm{(iv)} \quad & &
  x = -\frac{\alpha _1 (\alpha _1 + \alpha _2)}{z_0 (z_0 -1)} \cdot T + O\big(T^2\big),\qquad  y = \frac{z_0 (z_0-1)}{\alpha _1} \cdot T^{-1} + O(1), &\\
& \mathrm{(v)} \quad & &
 x = \frac{\alpha _1 \alpha _2}{z_0 (z_0 -1)} \cdot T + O\big(T^2\big),\qquad  y = -\frac{z_0 (z_0-1)}{\alpha _1} \cdot T^{-1} + O(1), &
\end{alignat*}
where $A_2$ is an arbitrary constant and $B_2$ is a certain function of $A_2$.
Since $(p,q) = (1,0)$, (i),~(ii) and~(iii) are regular, while~(iv) and~(v) are exceptional.
The Kovalevskaya exponents of all series are $\kappa = 2$.

For all Painlev\'{e} equations, the number of types of Laurent series solutions is
smaller than the number of local charts of the space of initial conditions by one, see Table~\ref{table1}.

\section[The third Painlev\'{e} equation of type $D_6$]{The third Painlev\'{e} equation of type $\boldsymbol{D_6}$}\label{section3}

\subsection[$\text{P}_{\text{III} (D_6)}$ on $\mathbb C P^3(0,1,2,1)$]{$\boldsymbol{\text{P}_{\text{III} (D_6)}}$ on $\boldsymbol{\mathbb C P^3(0,1,2,1)}$}\label{section3.1}

The orbifold structure of $\mathbb C P^3(0,1,2,1)$ is given by
\begin{gather*}
\mathbb C P^3(0,1,2,1)  =  \mathbb C \times \mathbb C P^2(1,2,1)
 =  \mathbb C \times \big( \mathbb C^2 \cup \mathbb C^2 / \mathbb Z_2 \cup \mathbb C^2\big)
 =  \mathbb C^3 \cup \big(\mathbb C \times \mathbb C^2/\mathbb Z_2 \big) \cup \mathbb C^3.
\end{gather*}
Thus, the space is covered by three inhomogeneous coordinates $(X_2, Z_2, \varepsilon _2),
(X_3, Y_3, \varepsilon _3)$ and $(x,y,z)$ related as
\begin{gather}
x = X_2 =X_3, \qquad
y = \varepsilon _2^{-1} = Y_3 \varepsilon _3^{-1}, \qquad
z = Z_2 \varepsilon _2^{-2} = \varepsilon _3^{-2}
\label{3-1}
\end{gather}
(the f\/irst local chart $(Y_1, Z_1, \varepsilon _1)$ does not appear because $p=0$).

We give the third Painlev\'{e} equation of type $D_6$ on the local chart $(x,y,z)$.
On the other local charts, $(\text{P}_{\text{III} (D_6)})$ is expressed as
\begin{gather*}
  \frac{dX_2}{d\varepsilon _2}
 = \frac{2X_2 - 2X_2^2 + \varepsilon _2 (\alpha - \alpha X_2 - \beta X_2)}
 {\varepsilon _2 (1 -2X_2 - Z_2 - \alpha \varepsilon _2 - \beta \varepsilon _2)}, \\
  \frac{dZ_2}{d\varepsilon _2} =
 \frac{2Z_2 - 4X_2Z_2 - 2Z_2^2 + \varepsilon _2Z_2 (1-2\alpha - 2\beta)}
 {\varepsilon _2 (1 -2X_2 - Z_2 - \alpha \varepsilon _2 - \beta \varepsilon _2)}, \\
 \frac{dX_3}{d\varepsilon _3} = -\frac{2}{\varepsilon _3^2}
 \left( 2X_3Y_3 - 2X_3^2Y_3 + \varepsilon _3 (\alpha - \alpha X_3 - \beta X_3) \right), \\
 \frac{dY_3}{d\varepsilon _3} = -\frac{2}{\varepsilon _3^2}
 \left( 1 - Y_3^2 + 2X_3Y_3^2 + \varepsilon _3Y_3 \left(-\frac{1}{2} + \alpha + \beta\right) \right).
\end{gather*}
In order to apply dynamical systems theory later,
it is convenient to rewrite them as $3$-di\-men\-sio\-nal autonomous vector f\/ields of the form
\begin{gather}
  \dot{X}_2 = 2X_2 - 2X_2^2 + \varepsilon _2 (\alpha - \alpha X_2 - \beta X_2), \nonumber\\
  \dot{Z}_2 = 2Z_2 - 4X_2Z_2 - 2Z_2^2 + \varepsilon _2Z_2 (1-2\alpha - 2\beta), \nonumber\\
  \dot{\varepsilon }_2 =\varepsilon _2 (1 -2X_2 - Z_2 - \alpha \varepsilon _2 - \beta\varepsilon _2),
\label{3-4}
\end{gather}
and
\begin{gather}
  \dot{X}_3 = 2X_3Y_3 - 2X_3^2Y_3 + \varepsilon _3 (\alpha - \alpha X_3 - \beta X_3),\nonumber \\
  \dot{Y}_3 = 1 - Y_3^2 + 2X_3Y_3^2 + \varepsilon _3Y_3 (-\frac{1}{2} + \alpha + \beta), \nonumber\\
  \dot{\varepsilon }_3 = -\frac{\varepsilon _3^2}{2},
\label{3-5}
\end{gather}
where $(\dot{\,\,}) = d/dt$ and $t\in \mathbb C$ parameterizes each integral curve.

We also have the decomposition
\begin{gather*}
\mathbb C P^3(0,1,2,1)  =  \mathbb C P^2(0,1,2) \cup \mathbb C^3 \qquad (\text{disjoint}).
\end{gather*}
Since $\mathbb C P^2(0,1,2) = \mathbb C \times \mathbb C P^1(1,2)$ and $\mathbb C P^1(1,2)$ is the usual projective line,
\begin{gather*}
\mathbb C P^3(0,1,2,1) =  \big(\mathbb C \times \mathbb C P^1\big) \cup \mathbb C^3,
\end{gather*}
where $\mathbb C^3 = \{ (x,y,z)\}$ and $\mathbb C \times \mathbb C P^1 = \{ (X_2, Z_2, 0) \} \cup \{ (X_3, Y_3, 0) \}$
in coordinates.
This implies that the set $\mathbb C \times \mathbb C P^1$ is attached at inf\/inity of the natural phase space
$\mathbb C^3 = \{ (x,y,z)\}$ of $(\text{P}_{\text{III} (D_6)})$, and the asymptotic behavior of solutions
can be studied by the limit $\varepsilon _2 \to 0$ or $\varepsilon _3 \to 0$.

The vector f\/ields~(\ref{3-4}) and~(\ref{3-5}) have exactly four f\/ixed points on the inf\/inity set
$\{ \varepsilon _2 = 0\} \cup \{ \varepsilon _3 =0 \}$ given by
\begin{gather*}
(X_2, Z_2, \varepsilon _2) = (0,0,0),\ (1,0,0),\  (0,1,0), \  (1,-1,0).
\end{gather*}
The Jacobi matrices of (\ref{3-4}) at the f\/ixed points are
\begin{gather}
\left(
\begin{matrix}
2 & 0 & \alpha \\
0 & 2 & 0 \\
0 & 0 & 1
\end{matrix}
\right), \quad -\left(
\begin{matrix}
2 & 0 & \beta \\
0 & 2 & 0 \\
0 & 0 & 1
\end{matrix}
\right), \quad
\left(
\begin{matrix}
2 & 0 & \alpha \\
-4 & -2 & \gamma \\
0 & 0 & 0
\end{matrix}
\right), \quad
-\left(
\begin{matrix}
2 & 0 & \beta \\
-4 & -2 & \gamma \\
0 & 0 & 0
\end{matrix}
\right),
\label{3-8}
\end{gather}
respectively, where $\gamma = 1-2\alpha -2\beta$.
The latter two points, for which $Z_2 \neq 0 \Rightarrow z = \infty$, correspond to the
irregular singular point.
Since the Jacobi matrix has a zero eigenvalue, there exists a one-dimensional center manifold
at each f\/ixed point.
The asymptotic expansion of the center manifold~\cite{Cho} yields the asymptotic expansion
of $(x(z), y(z))$ as~$|z| \to \infty$.

On the other hand, the former two f\/ixed points correspond to movable poles.
Indeed, it is easy to verify by using (\ref{3-1}) that the regular Laurent series solutions~(ii) and~(i) given in Section~\ref{section2.2} converge to the points $(X_2, Z_2, \varepsilon _2) = (0,0,0)$ and $(1,0,0)$,
respectively, as $z \to z_0$.
The exceptional Laurent series solution (iii) does not converge to some point on $\mathbb C P^3(0,1,2,1)$
as $z \to z_0$; the space $\mathbb C P^3(0,1,2,1)$ is not compact.

To treat the Laurent series (iii), we use the B\"{a}cklund transformations.
See \cite{Tsu} for the complete list of the B\"{a}cklund transformations of two-dimensional Painlev\'{e} equations.
It is known that the transformation groups of the Painlev\'{e} equations are
isomorphic to the extended af\/f\/ine Weyl groups.
For a classical root system $R$, the af\/f\/ine Weyl group and the extended af\/f\/ine Weyl group are denoted by
$W(R^{(1)})$ and $\widetilde{W}(R^{(1)})$, respectively.
Let $G = \operatorname{Aut} (R^{(1)})$ be the Dynkin automorphism group of the extended Dynkin diagram.
We have $\widetilde{W}(R^{(1)}) \cong G \ltimes W(R^{(1)})$.
For the third Painlev\'{e} equation of type $D_6$, $R = 2A_1$ and
\begin{gather*}
  \widetilde{W}\big((2A_1)^{(1)}\big) \cong G \ltimes W\big((2A_1)^{(1)}\big), \qquad
  W\big((2A_1)^{(1)}\big) = \langle s_0, s_1, s_0', s_1' \rangle, \\
  G = \operatorname{Aut} \big(D_6^{(1)}\big) = \operatorname{Aut} \big((2A_1)^{(1)}\big) = \langle \pi_1, \pi_2, \sigma _1 \rangle.
\end{gather*}
The action of each element is given in Table~\ref{table2}, where
$f_0 = x + (\alpha + \beta -1)/y + z/y^2$, $f_1 = \beta y + (x-1)y^2 + z$, $f_2 = \alpha y + xy^2 + z$ and $f_3 = 2xy + \alpha + \beta -1$.
\begin{table}[h]
\centering
\begin{tabular}{|c||c|c|c|c|c|}
\hline
 & $\alpha$ & $\beta$ & $y$ & $x$ & $z$ \\ \hline \hline
& & & & & \\[-0.4cm]
$s_0$ & \!\!\! $2-\alpha $ \!\!\! & $\beta$ &
 $\displaystyle y + \frac{1-\alpha}{f_0-1}$ &
 $\displaystyle x - \frac{(1-\alpha) (f_3-2y)}{f_1}-\frac{(1-\alpha)^2 z}{f_1^2}$ & $z$ \\[0.3cm] \hline
& & & & & \\[-0.4cm]
$s_1$ & $-\alpha$ & $\beta$ & $\displaystyle y+\frac{\alpha}{x}$ & $x$ & $z$ \\[0.3cm] \hline
& & & & & \\[-0.4cm]
$s_0'$ &$\alpha $ & \!\!\!$2-\beta$\!\!\! &
 $\displaystyle y + \frac{1-\beta}{f_0}$ &
 $\displaystyle x - \frac{(1-\beta) f_3}{f_2}-\frac{(1-\beta)^2 z}{f_2^2}$ & $z$ \\[0.3cm] \hline
& & & & & \\[-0.4cm]
$s_1'$ &$\alpha$ & $-\beta$ & $\displaystyle y + \frac{\beta}{x-1}$ & $x$ & $z$ \\[0.3cm] \hline
& & & & & \\[-0.4cm]
$\pi_1$ &\!\!\!$1-\alpha$\!\!\! & $\beta$ & $\displaystyle -\frac{z}{y}$ &
 $\displaystyle \frac{y}{z} (xy-y+\beta) + 1$ & $z$ \\[0.3cm] \hline
& & & & & \\[-0.4cm]
$\pi_2$ &$\alpha$ & \!\!$1-\beta$\!\! & $\displaystyle \frac{z}{y}$ &
 $\displaystyle -\frac{y}{z} (xy + \alpha)$ & $z$ \\[0.3cm] \hline
$\sigma _1$ &$\beta$ & $\alpha$ & $-y$ & $1-x$ & $-z$ \\ \hline
\end{tabular}
\caption{The action of the extended af\/f\/ine Weyl group for $(\text{P}_{\text{III} (D_6)})$.}\label{table2}
\end{table}

Let us consider another space $\mathbb C P^3(0,1,2,1)$ with inhomogeneous coordinates
$(\widetilde{X}_2, \widetilde{Z}_2, \widetilde{\varepsilon }_2)$, $(\widetilde{X}_3, \widetilde{Y}_3, \widetilde{\varepsilon }_3)$
and $(\widetilde{x}, \widetilde{y}, \widetilde{z})$ satisfying the same formula as (\ref{3-1}).
We glue two copies of $\mathbb C P^3(0,1,2,1)$ by the transformation $\pi_2$
\begin{gather}
(\widetilde{x}, \widetilde{y}, \widetilde{z}, \widetilde{\alpha }, \widetilde{\beta})
 =  \pi_2(x,y,z,\alpha ,\beta)
 =  \left( -\frac{y}{z}(xy + \alpha ) , \frac{z}{y},  z,  \alpha ,  1-\beta \right),
\label{3-9}
\end{gather}
which def\/ines a manifold denoted by $\mathcal{M}_1$.
Since $(\text{P}_{\text{III} (D_6)})$ is invariant under the action of~$\pi_2$,
the system written in the $(\widetilde{X}_2, \widetilde{Z}_2, \widetilde{\varepsilon }_2)$-chart
also satisf\/ies equation~(\ref{3-4}), in which $(\alpha , \beta)$ is replaced by $(\widetilde{\alpha }, \widetilde{\beta})
 = (\alpha , 1-\beta)$.
By the formula~(\ref{3-9}) or
\begin{gather*}
\widetilde{X}_2 = -\big(xy^2 + \alpha y\big)/z, \qquad
\widetilde{Z}_2 = y^2/z, \qquad
\widetilde{\varepsilon }_2 = y/z,
\end{gather*}
it is easy to show that the exceptional Laurent series (iii) is converted to the regular series (i)
in the $(\widetilde{x}, \widetilde{y}, \widetilde{z})$-chart,
and it approaches to the f\/ixed point $(\widetilde{X}_2, \widetilde{Z}_2, \widetilde{\varepsilon }_2) = (1,0,0)$
as $z\to z_0$.

We have proved that Laurent series solutions (i), (ii) and (iii) of $(\text{P}_{\text{III} (D_6)})$
converge to the f\/ixed points $(X_2, Z_2, \varepsilon _2) = (1,0,0), (0,0,0)$ and
$(\widetilde{X}_2, \widetilde{Z}_2, \widetilde{\varepsilon }_2) = (1,0,0)$ on $\mathcal{M}_1$, respectively, as $z\to z_0$.
Hence, the study of movable poles are reduced to the study of the f\/ixed points and
dynamical systems theory is applicable to investigate them.

As a simple application, we can prove the next theorem.

\begin{Theorem}
There exists a local analytic transformation $(X_2, Z_2, \varepsilon _2) \mapsto (u,v,w)$
defined near $(X_2, Z_2, \varepsilon _2) = (0,0,0)$ such that equation~\eqref{3-4} is transformed into the
linearized system
\begin{gather*}
\dot{u} = 2u + \alpha w, \qquad
\dot{v} = 2v, \qquad
\dot{w} = w.
\end{gather*}
Similarly, the vector field is locally linearized around $(X_2, Z_2, \varepsilon _2) = (1,0,0)$
and $(\widetilde{X}_2, \widetilde{Z}_2, \widetilde{\varepsilon }_2) = (1,0,0)$.
\end{Theorem}

This result implies that $(\text{P}_{\text{III} (D_6)})$ is locally transformed to the integrable system
around each movable singularity.
A proof is a straightforward application of Poincar\'{e}'s linearization theorem~\cite{Cho}
in dynamical systems theory.
See~\cite{Chi1} for the detail, in which a similar result is proved for the f\/irst, second and fourth
Painlev\'{e} equations, and also~\cite{Chi2}, in which it is proved that any dif\/ferential equations
having the Painlev\'{e} property is locally linearizable.
Such a local integrable system for the Painlev\'{e} equation is called the singular normal form in \cite{Chi1, CosCos}.

\subsection{The space of initial conditions}\label{section3.2}

A purpose in this section is to construct the space of initial conditions for $(\text{P}_{\text{III} (D_6)})$.
On the manifold $\mathcal{M}_1$, there are three singularities of the foliation of integral curves;
$(X_2, Z_2, \varepsilon _2) = (1,0,0), (0,0,0)$
and $(\widetilde{X}_2, \widetilde{Z}_2, \widetilde{\varepsilon }_2) = (1,0,0)$.
They correspond to movable poles of the Laurent series solutions (i), (ii) and (iii), respectively.
We will resolve these singularities by weighted blow-ups.
On the blow-up space, $(\text{P}_{\text{III} (D_6)})$ is again written in a Hamiltonian system,
whose Hamiltonian function is polynomial in dependent variables.
This implies that singularities of the foliation are resolved and the space of initial conditions
is obtained by three times blow-ups of $\mathcal{M}_1$.
This strategy is applicable to the other Painlev\'{e} equations, even for higher-dimensional Painlev\'{e} equations.

(i) blow-up at $(X_2, Z_2, \varepsilon _2) = (0,0,0)$.

For the vector f\/ield (\ref{3-4}), put
\begin{gather*}
u = X_2 + \alpha \varepsilon _2, \qquad
v = Z_2, \qquad
w = \varepsilon _2.
\end{gather*}
Then, the linear part is diagonalized and we obtain
\begin{gather*}
\dot{u} = 2u - 2u^2 + w(\alpha u - \alpha v - \beta u), \\
\dot{v} = 2v + v(w-4u-2v+2\alpha w-2\beta w), \\
\dot{w} = w - w(2u + v - \alpha w + \beta w).
\end{gather*}
The origin $(u,v,w) = (0,0,0)$ is a singularity of the foliation of integral curves.
To resolve it, we introduce the weighted blow-up def\/ined by
\begin{gather*}
u   = u_1^2   = v_2^2 u_2   = w_3^2 u_3, \qquad
v   = u_1^2v_1   = v_2^2   = w_3^2 v_3, \qquad
w   = u_1w_1   = v_2 w_2   = w_3.
\end{gather*}
The weight $(2,2,1)$, the exponents in the right hand sides, is taken from the eigenvalues
of the Jacobi matrix at the singularity.
The exceptional divisor $\{ u_1 = 0\} \cup \{ v_2 = 0\} \cup \{ w_3 = 0\}$
is the weighted projective space $\mathbb C P^2 (2,2,1)$.
The relation between the original coordinates $(x,y,z)$ and the new coordinates $(u_3, v_3, w_3)$ is given by
\begin{alignat}{4}
& x = u_3 w_3^2 - \alpha w_3, \qquad&&
y = 1/w_3, \qquad&&
z = v_3,&\nonumber\\
&u_3 = xy^2 + \alpha y, \qquad&&
w_3 = 1/y, \qquad&&
v_3 = z.&
\label{3-13}
\end{alignat}
Note that eventually the independent variable $z$ is not changed and
(\ref{3-13}) def\/ines a f\/iber bundle over $z$-space, whose f\/iber is a $(x,y)$-space
and $(u_3, w_3)$-space glued by the above relation.
In the new coordinates, $(\text{P}_{\text{III} (D_6)})$ is transformed into the Hamiltonian system
\begin{gather*}
\frac{dw_3}{dz} = -\frac{\partial H_1}{\partial u_3}, \qquad \frac{du_3}{dz} = \frac{\partial H_1}{\partial w_3},
\end{gather*}
with the Hamiltonian function
\begin{gather*}
zH_1 = -u + u^2w^2 + uw^2z - \alpha wz - (\alpha -\beta ) uw
\end{gather*}
(here, the subscript is omitted for simplicity).
Since $H_1$ is polynomial in $(u_3, w_3)$, there are no singularities of the foliation in this chart;
the singularity associated with the Laurent series solution~(ii) is resolved.
Furthermore, we can verify that
\begin{gather}
dx \wedge dy= dw_3 \wedge du_3,\qquad dx \wedge dy + dH \wedge dz = dw_3 \wedge du_3 + dH_1 \wedge dz.
\label{3-15}
\end{gather}

(ii) blow-up at $(X_2, Z_2, \varepsilon _2) = (1,0,0)$.

Putting $X_2 -1 = \hat{X}_2$ and $u = \hat{X}_2 + \beta \varepsilon _2$, $v = Z_2$, $w = \varepsilon _2$
for the vector f\/ield~(\ref{3-4}) results in
\begin{gather*}
\dot{u} = -2u - 2u^2 - w(\alpha u - \beta u + \beta v), \\
\dot{v} = -2v + v(w-4u-2v-2\alpha w+2\beta w), \\
\dot{w} = -w - w(2u + v + \alpha w - \beta w).
\end{gather*}
To resolve the singularity at $(u,v,w) = (0,0,0)$, we introduce the weighted blow-up with the weight $(2,2,1)$
\begin{gather*}
u   = u_4^2   = v_5^2 u_5   = w_6^2 u_6, \qquad
v   = u_4^2v_4   = v_5^2   = w_6^2 v_6, \qquad
w   = u_4w_4   = v_5 w_5   = w_6.
\end{gather*}
The relation between the original coordinates $(x,y,z)$ and the new coordinates $(u_6, v_6, w_6)$ is given by
\begin{alignat}{4}
& x = u_6 w_6^2 - \beta w_6 + 1, \qquad &&
y = 1/w_6, \qquad &&
z = v_6,& \nonumber\\
& u_6 = xy^2 + \beta y - y^2, \qquad&&
w_6 = 1/y, \qquad &&
v_6 = z,&
\label{3-17}
\end{alignat}
which def\/ines a f\/iber bundle over $z$-space.
In the new coordinates, $(\text{P}_{\text{III} (D_6)})$ is written as the Hamiltonian system
\begin{gather*}
\frac{dw_6}{dz} = -\frac{\partial H_2}{\partial u_6}, \qquad \frac{du_6}{dz} = \frac{\partial H_2}{\partial w_6},
\end{gather*}
with the Hamiltonian function
\begin{gather}
zH_2 = u + u^2w^2 + uw^2z - \beta wz + (\alpha -\beta ) uw
\label{3-18}
\end{gather}
(here, the subscript is omitted for simplicity).
Since $H_2$ is polynomial in $(u_6, w_6)$, the singularity associated with the Laurent series solution (i) is resolved.
As before, we can verify the symplectic relation~(\ref{3-15}).

(iii) blow-up at $(\widetilde{X}_2, \widetilde{Z}_2, \widetilde{\varepsilon} _2) = (1,0,0)$.

This singularity is resolved by the same way as above by the weighted blow-up with the weight~$(2,2,1)$.
The f\/inal result is easily obtained by the B\"{a}cklund transformation~(\ref{3-9}) as follows.
Def\/ine the new coordinates $(u_9, v_9, w_9)$ by
\begin{gather*}
u_9 = \widetilde{x}\widetilde{y}^2 + \widetilde{\beta} \widetilde{y} - \widetilde{y}^2, \qquad
w_9 = 1/\widetilde{y}, \qquad
v_9 = \widetilde{z},
\end{gather*}
which is the same formula as~(\ref{3-17}).
Substituting~(\ref{3-9}) yields
\begin{gather}
u_9 = -z^2/y^2 - xz - (\alpha +\beta -1) z/y, \qquad
w_9 = y/z, \qquad
v_9 = z.
\label{3-20}
\end{gather}
By this coordinates change, $(\text{P}_{\text{III} (D_6)})$ is transformed into the
Hamiltonian system of $(u_9, w_9)$, whose Hamiltonian function has the same form as~(\ref{3-18}),
though $(\alpha , \beta)$ is replaced by $(\widetilde{\alpha }, \widetilde{\beta}) = (\alpha , 1-\beta)$.

In this manner, all singularities are resolved and we have

\begin{Theorem}
The space of initial conditions $E(z)$ of $({\rm P}_{{\rm III} (D_6)})$ is given by
$\mathbb C^2_{(x,y)} \cup \mathbb C^2_{(u_3,w_3)} \cup \mathbb C^2_{(u_6, w_6)} \cup \mathbb C^2_{(u_9,w_9)}$ glued by
the symplectic transformations~\eqref{3-13}, \eqref{3-17} and~\eqref{3-20}.
The space~$E(z)$ is a nonsingular symplectic surface parameterized by $z\in \mathbb C\backslash \{ 0\}$,
on which $({\rm P}_{{\rm III} (D_6)})$ is expressed as a polynomial Hamiltonian system.
\end{Theorem}

\subsection{The Riccati solutions}\label{section3.3}

It is known that when the parameter $\alpha $ is an integer or $\beta$ is an integer,
there exists a one-parameter family of solutions of $(\text{P}_{\text{III} (D_6)})$ satisfying
the Riccati equation, which is equivalent to the Bessel equation.
For example, when $\alpha =0$ (resp. $\beta = 0$), the Riccati equation is given by
\begin{gather}
z \frac{dy}{dz} = -y^2 + \beta y + z, \qquad \text{resp.} \quad z \frac{dy}{dz} = y^2 + \alpha y + z.
\label{Riccati}
\end{gather}
The B\"{a}cklund transformations yield the Riccati equations for a general case $\alpha \in \mathbb Z$
or $\beta \in \mathbb Z$.
Let us prove this fact from a view point of dynamical systems theory.

Recall that there are two f\/ixed points $(X_2, Z_2, \varepsilon _2) = (0,1,0), (1,-1,0)$
of the vector f\/ield~(\ref{3-4}) corresponding to the irregular singular point.
The Jacobi matrices at these points are shown in~(\ref{3-8}),
and eigenvalues of them are $2$, $-2$, $0$.
Hence, there exist a~center-stable manifold and a center-unstable manifold at these points.
Let us calculate them explicitly.

For the point $(X_2, Z_2, \varepsilon _2) = (0,1,0)$, put $\hat{Z}_2 = Z_2 - 1$ and
\begin{gather*}
  u = X_2 + \hat{Z}_2 + \frac{1}{2}\varepsilon _2 (\alpha + 2 \beta -1), \qquad
  v = -X_2 - \frac{1}{2}\alpha \varepsilon _2, \qquad
  w = \varepsilon _2.
\end{gather*}
Then, the linear part of equation~(\ref{3-4}) is diagonalized around $(0,1,0)$ and we obtain
\begin{gather*}
 \dot{u} = -2u + u \left( -2u - \frac{1}{2}w - \frac{\alpha }{2}w + \beta w \right)
 + (1-\alpha ) \left( \frac{1}{2}vw + \frac{1}{4}w^2 - \frac{\beta}{2} w^2 \right), \\
 \dot{v} = 2v + v\left( 2v + \frac{\alpha }{2}w - \beta w \right)
 + \frac{1}{2}\alpha \left( uw + \frac{1}{2}w^2 - \beta w^2 \right), \\
  \dot{w} = w \left( v-u-\frac{1}{2}w \right).
\end{gather*}
The stable, unstable, center subspaces are $u,v,w$-directions, respectively.
In particular, the center-unstable manifold (resp. the center-stable manifold)
is expressed as the graph of some function $v = \phi (u,w)$ (resp.\ $u = \varphi (v,w)$).

It is obvious that when $\alpha = 0$, the center-unstable manifold is exactly given by $v = 0$.
This implies $X_2 = x = 0$.
When $\alpha = 0$ and $x = 0$, equation~(\ref{1-7}) is reduced to the Riccati equation~(\ref{Riccati}).
Similarly, when $\alpha =1$, the center-stable manifold is exactly given by $u=0$.
This implies
\begin{gather*}
0 = X_2 + \hat{Z}_2 + \beta \varepsilon _2 = x + z/y^2 + \beta /y -1.
\end{gather*}
Substituting this relation to equation~(\ref{1-7}) yields the Riccati equation for $\alpha =1$,
though it is also obtained by the B\"{a}cklund transformation to that for~$\alpha =0$.
The same argument at the point $(X_2, Z_2, \varepsilon _2) = (1,-1,0)$ provides
the Riccati equation for $\beta = 0$.

\begin{Theorem}
When $\alpha \in \mathbb Z$ or $\beta \in \mathbb Z$, there exists a one-parameter family of solutions
of $({\rm P}_{{\rm III} (D_6)})$ governed by the Riccati equation of Bessel type.
The one-parameter family of solutions forms a center-$($un$)$stable manifold of~\eqref{3-4}
in $\mathbb C P^3(0,1,2,1)$.
\end{Theorem}

The same argument is applicable to the other Painlev\'{e} equations, even for higher-dimensional
Painlev\'{e} equations.
For two-dimensional equations, it is well known that when parameters take certain specif\/ic values,
the Painlev\'{e} equations are reduced to Riccati-type equations except for~$(\text{P}_{\text{I}})$, $(\text{P}_{\text{III} (D_7)})$ and $(\text{P}_{\text{III} (D_8)})$.
For such specif\/ic values of parameters, center-(un)stable manifolds in $\mathbb C P^3(p,q,r,s)$
are exactly calculated and the Riccati equations are obtained by restricting the equations
to the center-(un)stable manifolds.

\looseness=-1
The result is summarized in Table~\ref{table3}.
The third column shows one of the parameters for which equations are reduced to the Riccati equations.
The other possible parameters are obtained by the B\"{a}cklund transformations.
The fourth column denotes the name of the Riccati equation when it is written as a second order linear equation.
Note that the weight of the Riccati equation is also def\/ined through the Newton polyhedron.
For example, the Riccati equation of Airy type is def\/ined by
\begin{gather*}
\frac{dy}{dz} = y^2 + z.
\end{gather*}
The exponents of monomials in the right hand side are $(1,1)$ and $(-1,2)$.
Since they are on the line $y + 2z = 3$, the weighted projective space for the equation
is $\mathbb C P^2 (1,2,3)$.
See~\cite{Chi3} for the analy\-sis of the Airy equation by means of the weighted projective space.
The last column of Table~\ref{table3} gives the weight of each Riccati equation.
It is interesting to note that these weights are obtained by deleting the f\/irst or second numbers
from the weights of the corresponding Painlev\'{e} equations.

\begin{table}[h]
\centering
\begin{tabular}{|c||c|c|c|c|}
\hline
 & weight & parameter & Riccati & weight \\ \hline \hline
$\text{P}_\text{II}$ & $(2,1,2,3)$ & $\alpha =1/2$ & Airy & $(1,2,3)$ \\ \hline
$\text{P}_\text{IV}$ & $(1,1,1,2)$ & $\alpha =0$ & Hermite & $(1,1,2)$ \\ \hline
$\text{P}_{\text{III}(D_6)}$&$(0,1,2,1)$ & $\alpha =0$ & Bessel & $(1,2,1)$ \\ \hline
$\text{P}_\text{V}$ & $(1,0,1,1)$ & $\alpha_1 =0$ & CHG & $(1,1,1)$ \\ \hline
$\text{P}_\text{VI}$ & $(1,0,0,1)$ & $\alpha_1 =0$ & HG & $(1,0,1)$ \\ \hline
\end{tabular}

\caption{The type of Riccati equations and their weights.
CHG and HG denote the conf\/luent hyper\-geo\-metric and hyper\-geo\-metric equations, respectively.}\label{table3}
\end{table}

\subsection{Boutroux's coordinates}\label{section3.4}

For the f\/irst and second Painlev\'{e} equations, the third local chart $(X_3, Y_3, \varepsilon _3)$
of $\mathbb C P^3(p,q,r,s)$ def\/ined by equation~(\ref{2-3}) is equivalent to Boutroux's coordinates introduced in~\cite{Bou}
to investigate the irregular singular point $z = \infty$.
For the other Painlev\'{e} equations, the chart $(X_3, Y_3, \varepsilon _3)$ plays the same role
as Boutroux's coordinates;
as $\varepsilon _3 \to 0$, equation~(\ref{3-5}) is reduced to the autonomous Hamiltonian system
\begin{gather}
\dot{X}_3 = 2X_3Y_3 - 2X_3^2Y_3, \qquad
\dot{Y}_3 = 1 - Y_3^2 + 2X_3Y_3^2,
\label{auto}
\end{gather}
which is often called the autonomous limit.
Since a generic integral curve given by $\mathcal{H}_{{\rm III}(D_6)} = X_3^2 Y_3^2 - X_3Y_3^2 + X_3 = \text{const}$
is an elliptic curve, a general solution can be expressed by Weierstrass's elliptic functions.
Then, the system~(\ref{3-5}) with small $\varepsilon _3$ can be studied by a perturbation method.

Let us calculate the action of the extended af\/f\/ine Weyl group $\widetilde{W}\big((2A_1)^{(1)}\big)$
restricted on the set $\{ \varepsilon _3 = 0\}$, which leaves the autonomous limit (\ref{auto}) invariant.

\begin{Proposition}
The birational transformation group $\widetilde{W}\big((2A_1)^{(1)}\big)$ is extended to the
birational transformation group acting on $\mathbb C P^3(0,1,2,1)$.
The transformation group which leaves the autonomous limit \eqref{auto} invariant is given by
$\mathbb Z_2 \ltimes \operatorname{Aut} \big((2A_1)^{(1)}\big) =\mathbb Z_2 \ltimes \langle \pi_1, \pi_2, \sigma _1 \rangle$.
\end{Proposition}

\begin{proof}
The f\/irst part of Proposition is verif\/ied by a straightforward calculation.
To show the second part, we should write down the actions in the $(X_3, Y_3, \varepsilon _3)$-chart.
For example, the action of $s_1$ in the $(X_3, Y_3, \varepsilon _3)$-chart is given by
\begin{gather*}
(X_3 , Y_3, \varepsilon _3, \alpha , \beta) \mapsto
\left(X_3, Y_3 + \alpha \frac{\varepsilon _3}{X_3}, \varepsilon _3, -\alpha , \beta\right).
\end{gather*}
On the set $\{ \varepsilon _3 = 0\}$, it is reduced to
\begin{gather*}
(X_3 , Y_3, \alpha , \beta) \mapsto (X_3, Y_3, -\alpha , \beta).
\end{gather*}
Since the autonomous limit (\ref{auto}) is independent of the parameters $\alpha$, $\beta$,
the action of $s_1$ to~(\ref{auto}) is trivial.
Similarly, the actions of $s_0, s_0'$ and $s_1'$ to (\ref{auto}) are reduced to the trivial one.
On the other hand, it is easy to conf\/irm that the restriction of the actions of $\pi_1$, $\pi_2$, $\sigma _1$
are not trivial, which are explicitly given by
\begin{gather*}
 \pi_1 \colon \  (X_3, Y_3) \mapsto \big(X_3Y_3^2 - Y_3^2 + 1,  -1/Y_3\big), \\
  \pi_2 \colon \  (X_3, Y_3) \mapsto \big({-}X_3Y_3^2 ,  1/Y_3\big), \\
  \sigma _1 \colon \  (X_3, Y_3) \mapsto \big(1-X_3,  \sqrt{-1} Y_3\big).
\end{gather*}
Furthermore,~(\ref{auto}) is invariant under the $\mathbb Z_2$ action $Y_3 \mapsto -Y_3$
due to the orbifold structure of $\mathbb C P^3 (0,1,2,1)$.
Since $r = \deg (z) = 2$, $(X_3, Y_3, \varepsilon _3)$ are coordinates
on the lift $\mathbb C^3$ of the quotient $\mathbb C^3/\mathbb Z_2$, where $\mathbb Z_2$ action is def\/ined by
$(X_3, Y_3, \varepsilon _3) \mapsto (X_3, -Y_3, -\varepsilon _3)$.
This is reduced to the action $Y_3 \mapsto -Y_3$ on the set $\{ \varepsilon _3 = 0\}$.
\end{proof}

A similar result also holds for the other Painlev\'{e} equations except for ($\text{P}_{\text{VI}}$)
and summarized in Table~\ref{table4}.
If $r = \deg (z) = 1$, the symmetry group of the autonomous limit is given by
the Dynkin automorphism group $G = \operatorname{Aut} \big(R^{(1)}\big)$ because the action of
$W\big(R^{(1)}\big)$ on the set $\{ \varepsilon _3 = 0\}$ is reduced to the trivial action as above.
If $r>1$, the autonomous limit is further invariant under the $\mathbb Z_r$ action arising from
the orbifold structure of $\mathbb C P^3 (p,q,r,s)$.
The autonomous limit on the set $\{ \varepsilon _3 = 0\}$ is not def\/ined for ($\text{P}_{\text{VI}}$) in this way
because $r=0$.
\begin{table}[h]
\centering
\begin{tabular}{|c||c|c|c|}
\hline
 & weight & $\mathcal{H}_J$ & symmetry \\ \hline \hline
$\text{P}_\text{I}$ & $(3,2,4,5)$ & $X^2 - 4Y^3 - 2Y$ & $\mathbb Z_4$ \\ \hline
$\text{P}_\text{II}$ & $(2,1,2,3)$ & $X^2 - Y^4 - Y^2$ & $\mathbb Z_2\ltimes \operatorname{Aut} \big(A_1^{(1)}\big)$ \\ \hline
$\text{P}_\text{IV}$ & $(1,1,1,2)$ & $X^2Y - XY^2 - 2XY$ & $\operatorname{Aut} \big(A_2^{(1)}\big)$ \\ \hline
$\text{P}_{\text{III}(D_8)}$ & $(-1,2,4,1)$ & $X^2 Y^2 - \frac{1}{2}Y - \frac{1}{2Y}$ & $\mathbb Z_4\ltimes \mathbb Z_2$\\ \hline
$\text{P}_{\text{III}(D_7)}$ & $(-1,2,3,1)$ & $X^2Y^2 + X+Y$ & $\mathbb Z_3 \ltimes \operatorname{Aut} \big(A_1^{(1)}\big)$ \\ \hline
$\text{P}_{\text{III}(D_6)}$ & $(0,1,2,1)$ & $X^2 Y^2 - XY^2 + X$ & $\mathbb Z_2 \ltimes \operatorname{Aut} \big((2A_1)^{(1)}\big)$ \\ \hline
$\text{P}_{\text{V}}$ & $(1,0,1,1)$ & $X^2Y^2-X^2Y+XY^2-XY$ & $\operatorname{Aut} \big(A_3^{(1)}\big)$ \\ \hline
\end{tabular}
\caption{Hamiltonian functions of the autonomous limit
def\/ined on the set $\{ \varepsilon _3 = 0\}$.
$(X_3, Y_3)$ is denoted by $(X,Y)$ for simplicity.}\label{table4}
\end{table}

In the rest of this paper, the other Painlev\'{e} equations $(\text{P}_{\text{III}(D_7)})$,
$(\text{P}_{\text{III}(D_8)})$, $(\text{P}_{\text{V}})$ and $(\text{P}_{\text{VI}})$ are studied
with the aid of the weighted projective spaces and dynamical systems theory
(see~\cite{Chi1} for $(\text{P}_{\text{I}})$, $(\text{P}_{\text{II}})$ and $(\text{P}_{\text{IV}})$).
Since the strategy is completely the same as that for $(\text{P}_{\text{III}(D_6)})$,
we only show important steps and formulae.

\section[The third Painlev\'{e} equation of type $D_7$]{The third Painlev\'{e} equation of type $\boldsymbol{D_7}$}\label{section4}

The space $\mathbb C P^3(-1{,}2{,}3{,}1)$ for $(\text{P}_{\text{III}(D_7)})$ is covered by four inhomogeneous coordinates
$(Y_1{,} Z_1{,}\varepsilon _1)$, $ (X_2, Z_2, \varepsilon _2)$, $(X_3, Y_3, \varepsilon _3)$ and $(x,y,z)$ related as
\begin{gather*}
x =   \varepsilon _1 =   X_2\varepsilon _2 =  X_3\varepsilon _3,\qquad
y =   Y_1\varepsilon _1^{-2} =   \varepsilon _2^{-2}=  Y_3\varepsilon _3^{-2},\qquad
z =   Z_1\varepsilon_1 ^{-3} =   Z_2\varepsilon _2^{-3}=  \varepsilon _3^{-3}.
\label{4-1}
\end{gather*}

We give the third Painlev\'{e} equation of type $D_7$ on the local chart $(x,y,z)$.
On the other local charts, $(\text{P}_{\text{III} (D_7)})$ is expressed as rational dif\/ferential equations.
By rewriting them as $3$-dimensional autonomous vector f\/ields, we obtain
\begin{gather}
  \dot{Y}_1 = 2Y_1 + 2Y_1^2-Z_1+\alpha \varepsilon _1 Y_1,\nonumber \\
  \dot{Z}_1 =3Z_1 + 6Y_1Z_1 - \varepsilon _1Z_1 + 3\alpha \varepsilon _1Z_1,\label{4-2} \\
  \dot{\varepsilon }_1 = \varepsilon _1 (1 + 2Y_1 + \alpha \varepsilon _1),\nonumber
\\
 \dot{X}_2 = 2 + 2X_2^2 - X_2Z_2 + \alpha \varepsilon _2 X_2, \nonumber\\
 \dot{Z}_2 = 6X_2 Z_2+3Z_2^2 + 3\alpha \varepsilon _2Z_2-2\varepsilon _2 Z_2 ,\nonumber\\ 
 \dot{\varepsilon }_2 = \varepsilon _2 (2X_2 + Z_2 + \alpha \varepsilon _2),\nonumber
\end{gather}
and
\begin{gather}
\dot{X}_3 =-1-2X_3^2Y_3 + \varepsilon _3 \left(\frac{1}{3}X_3 - \alpha X_3\right) ,\nonumber \\
 \dot{Y}_3 = 1 + 2X_3Y_3^2 + \varepsilon _3 \left(-\frac{2}{3}Y_3 + \alpha Y_3\right), \label{4-4}\\
 \dot{\varepsilon }_3 = -\frac{1}{3}\varepsilon _3^2.\nonumber
\end{gather}
We have the decomposition
\begin{gather*}
\mathbb C P^3(-1,2,3,1)  =  \mathbb C P^2(-1,2,3) \cup \mathbb C^3 \qquad (\text{disjoint}),
\end{gather*}
where $\mathbb C^3 = \{ (x,y,z)\}$ and $\mathbb C P^2(-1,2,3) = \{ \varepsilon _1=0\}\cup \{ \varepsilon _2 = 0\} \cup
\{ \varepsilon _3 = 0\}$ in coordinates.
This implies that the set $\mathbb C P^2(-1,2,3)$ is attached at inf\/inity of the phase space~$\mathbb C^3$.
Thus, the asymptotic behavior of solutions can be studied by the limit
$\varepsilon _1 \to 0$ or $\varepsilon _2 \to 0$ or $\varepsilon _3 \to 0$.

The autonomous limit is a Hamiltonian system obtained by putting $\varepsilon _3 = 0$ for equation~(\ref{4-4}).
On the set $\varepsilon _3 = 0$, the action of the extended af\/f\/ine Weyl group
$\widetilde{W}(A_1^{(1)}) \simeq \operatorname{Aut} (A_1^{(1)}) \ltimes W(A_1^{(1)})$
is reduced to the action of $\operatorname{Aut} (A_1^{(1)}) \simeq \mathbb Z_2$ def\/ined by $(X_3, Y_3) \mapsto (Y_3, X_3)$.
Further, the autonomous limit is invariant under the~$\mathbb Z_3$ action~($r=3$) induced by the orbifold structure, see Table~\ref{table4}.

The vector f\/ield~(\ref{4-2}) has f\/ixed points on the inf\/inity set
$\mathbb C P^2(-1,2,3)$ given by $(Y_1, Z_1, \varepsilon _1) = (-1,0,0)$ and $(-1/2, -1/2,0)$.
The Jacobi matrices of~(\ref{4-2}) at the f\/ixed points are
\begin{gather*}
-\left(
\begin{matrix}
2 & 1 & \alpha \\
0 & 3 & 0 \\
0 & 0 & 1
\end{matrix}
\right), \quad -
\left(
\begin{matrix}
0 & 1 & \alpha /2 \\
3 & 0 & (3\alpha -1)/2 \\
0 & 0 & 0
\end{matrix}
\right),
\end{gather*}
respectively.
The latter f\/ixed point having a zero eigenvalue corresponds to the irregular singular point,
and the former one corresponds to a movable singularity associated with
the re\-gu\-lar Laurent series solution (i) given in Section~\ref{section2.2};
The Laurent series~(i) converges to the point $(Y_1, Z_1, \varepsilon _1) = (-1,0,0)$ as $z\to z_0$.
The exceptional Laurent series solution (ii) does not converge to some point on $\mathbb C P^3(-1,2,3,1)$ as $z \to z_0$.

\begin{Remark}
These f\/ixed points are also included in the chart $(X_2, Z_2, \varepsilon _2)$.
However, it is better to use the f\/irst chart $(Y_1, Z_1, \varepsilon _1)$ because
the second chart has to be divided by the $\mathbb Z_2$ action due to the orbifold structure.
\end{Remark}

To treat the Laurent series (ii), we use the B\"{a}cklund transformation $\sigma $ def\/ined by
\begin{gather}
(\widetilde{x}, \widetilde{y}, \widetilde{z}, \widetilde{\alpha })
= \sigma (x,y,z,\alpha ) = \left( -\frac{y}{z},  zx,  -z,  1-\alpha \right).
\label{4-6}
\end{gather}
We consider another space $\mathbb C P^3(-1,2,3,1)$ with inhomogeneous coordinates denoted by
$(\widetilde{x}, \widetilde{y}, \widetilde{z})$ etc.
We glue two copies of $\mathbb C P^3(-1,2,3,1)$ by the transformation~$\sigma $.
Then, it is easy to show that the exceptional Laurent series~(ii) is converted to the regular series~(i)
in the $(\widetilde{x}, \widetilde{y}, \widetilde{z})$-chart,
and it approaches to the f\/ixed point $(\widetilde{Y}_1, \widetilde{Z}_1, \widetilde{\varepsilon }_1) = (-1,0,0)$
as $z\to z_0$.

To construct the space of initial conditions of $(\text{P}_{\text{III} (D_7)})$,
we perform the weighted blow-ups at the points $(Y_1, Z_1, \varepsilon _1) = (-1,0,0)$ and
$(\widetilde{Y}_1, \widetilde{Z}_1, \widetilde{\varepsilon }_1) = (-1,0,0)$.

(i) blow-up at $(Y_1, Z_1, \varepsilon _1) = (-1,0,0)$.

For the vector f\/ield (\ref{4-2}), put $\hat{Y}_1 = Y_1+1$ and
\begin{gather*}
u = \hat{Y}_1 - Z_1 + \alpha \varepsilon _1, \qquad v = Z_1,\qquad w = \varepsilon _2.
\end{gather*}
Then, the linear part is diagonalized around $(-1,0,0)$ and we obtain
\begin{gather*}
\dot{u} = 2u - 2u^2 + 2uv + 4v^2 - vw + \alpha (uw - 2vw), \\
\dot{v} = 3v + v(w-6u-6v + 3\alpha w), \\
\dot{w} = w - w(2u + 2v - \alpha w).
\end{gather*}
The origin $(u,v,w) = (0,0,0)$ is a singularity of the foliation of integral curves.
To resolve it, we introduce the weighted blow-up def\/ined by
\begin{gather*}
u   = u_1^2   = v_2^2 u_2   = w_3^2 u_3, \qquad
v   = u_1^3v_1   = v_2^3    = w_3^3 v_3, \qquad
w   = u_1w_1   = v_2 w_2   = w_3.
\end{gather*}
The \looseness=-1 weight $(2,3,1)$ is taken from the eigenvalues of the Jacobi matrix at the singularity.
The relation between the original coordinates $(x,y,z)$ and the new coordinates $(u_3, v_3, w_3)$ is given by
\begin{alignat}{4}
& x = w_3, \qquad &&
y = u_3 + v_3w_3-\alpha /w_3-1/w_3^2, \qquad &&
z = v_3,&\nonumber\\
& u_3 = y-xz+\alpha /x+1/x^2, \qquad &&
w_3 = x, \qquad &&
v_3 = z. &
\label{4-8}
\end{alignat}
In the new coordinates, $(\text{P}_{\text{III} (D_7)})$ is transformed into the Hamiltonian system
\begin{gather*}
\frac{dw_3}{dz} = -\frac{\partial H_1}{\partial u_3}, \qquad \frac{du_3}{dz} = \frac{\partial H_1}{\partial w_3},
\end{gather*}
with the polynomial Hamiltonian function
\begin{gather}
zH_1 = -u + u^2w^2 - \frac{1}{2}w^2z + 2uw^3 z + w^4z^2 - \alpha \big(uw +w^2z\big)
\label{4-9}
\end{gather}
(here, the subscript is omitted for simplicity).
Furthermore, we can verify the symplectic relation~(\ref{3-15}).

(ii) blow-up at $(\widetilde{Y}_1, \widetilde{Z}_1, \widetilde{\varepsilon }_1) = (-1,0,0)$.

This singularity is resolved by the same way as above by the weighted blow-up with the weight $(2,3,1)$.
The result is easily obtained by the B\"{a}cklund transformation $\sigma $ as follows.
Def\/ine the new coordinates $(u_6, v_6, w_6)$ by
\begin{gather*}
u_6 = \widetilde{y}-\widetilde{x}\widetilde{z} + \widetilde{\alpha }/\widetilde{x}+1/\widetilde{x}^2, \qquad
w_6 = \widetilde{x}, \qquad
v_6 = \widetilde{z}.
\end{gather*}
Substituting (\ref{4-6}) yields
\begin{gather}
u_6 = xz + xy - (1-\alpha )z/y + z^2/y^2, \qquad
w_6 = -y/z, \qquad
v_6 = -z.
\label{4-11}
\end{gather}
By this coordinates change, $(\text{P}_{\text{III} (D_7)})$ is transformed into the
Hamiltonian system of $(u_6, w_6)$, whose Hamiltonian function has the same form as~(\ref{4-9}),
for which $\alpha$ is replaced by $\widetilde{\alpha } = 1-\alpha$.

In this manner, all singularities are resolved and
the space of initial conditions of $(\text{P}_{\text{III} (D_7)})$ is given by
$\mathbb C^2_{(x,y)} \cup \mathbb C^2_{(u_3,w_3)} \cup \mathbb C^2_{(u_6, w_6)}$ glued by
the symplectic transformations~(\ref{4-8}),~(\ref{4-11}).

\section[The third Painlev\'{e} equation of type $D_8$]{The third Painlev\'{e} equation of type $\boldsymbol{D_8}$}\label{section5}

\looseness=-1
The orbifold $\mathbb C P^3(-1,2,4,1)$ for $(\text{P}_{\text{III}(D_8)})$ is covered by four inhomogeneous coordinates
related as
\begin{gather*}
x =   \varepsilon _1 =   X_2\varepsilon _2 =  X_3\varepsilon _3, \qquad
y =   Y_1\varepsilon _1^{-2} =   \varepsilon _2^{-2}=  Y_3\varepsilon _3^{-2},\qquad
z =   Z_1\varepsilon_1 ^{-4} =   Z_2\varepsilon _2^{-4}=  \varepsilon _3^{-4}.
\end{gather*}
We give the third Painlev\'{e} equation of type $D_8$ on the local chart $(x,y,z)$.
On the other local charts, $(\text{P}_{\text{III} (D_8)})$ is expressed as rational dif\/ferential equations.
By rewriting them as $3$-dimensional autonomous vector f\/ields, we obtain
\begin{gather}
  \dot{Y}_1 = Y_1^3 - 2Y_1^4 - Y_1Z_1 , \nonumber\\
  \dot{Z}_1 = 2 Y_1^2Z_1 - 2Z_1^2-8Y_1^3Z_1 + \varepsilon _1 Y_1^2Z_1, \label{5-2}\\
  \dot{\varepsilon }_1 = \varepsilon _1 \left(\frac{1}{2}Y_1^2 - 2Y_1^3 - \frac{1}{2}Z_1\right),\nonumber\\
\dot{X}_3 = -2X_3^2 Y_3 + \frac{1}{2} - \frac{1}{2Y_3^2} + \frac{1}{4}\varepsilon _3X_3, \nonumber\\
 \dot{Y}_3 = 2X_3Y_3^2 - \frac{1}{2} \varepsilon _3Y_3, \label{5-3}\\
 \dot{\varepsilon }_3 = -\frac{1}{4}\varepsilon _3^2.\nonumber
\end{gather}
We will not use the vector f\/ield written in $(X_2, Z_2, \varepsilon _2)$-chart.

The autonomous limit is a Hamiltonian system obtained by putting $\varepsilon _3 = 0$ for equation~(\ref{5-3}).
It is known that $(\text{P}_{\text{III}(D_8)})$ is invariant under the transformation~$\pi$ def\/ined by
\begin{gather}
(\widetilde{x}, \widetilde{y}, \widetilde{z})
= \pi (x,y,z) = \left( -\frac{xy^2}{z} + \frac{y}{2z},  \frac{z}{y},  z \right).
\label{5-4}
\end{gather}
On the set $\varepsilon _3 = 0$, this action is reduced to the $\mathbb Z_2$ action given by
$(X_3, Y_3) \mapsto (-X_3Y_3^2,  1/Y_3)$.
Further, the autonomous limit is invariant under the $\mathbb Z_4$ action ($r=4$) induced by the orbifold structure, see Table~\ref{table4}.

The vector f\/ield~(\ref{5-2}) has the f\/ixed point $(Y_1, Z_1, \varepsilon _1) = (1/2,0,0)$.
The Jacobi matrix of~(\ref{5-2}) at the f\/ixed point is
\begin{gather*}
-\frac{1}{8} \left(
\begin{matrix}
2 & 4 & 0 \\
0 & 4 & 0 \\
0 & 0 & 1
\end{matrix}
\right).
\end{gather*}
The regular Laurent series solution (i) given in Section~\ref{section2.2}
converges to this point as $z\to z_0$, while
the exceptional Laurent series solution~(ii) does not converge to some point on $\mathbb C P^3(-1,2,4,1)$.
To treat the Laurent series (ii), consider another space $\mathbb C P^3(-1,2,4,1)$ with inhomogeneous coordinates denoted by
$(\widetilde{x}, \widetilde{y}, \widetilde{z})$ etc.
We glue two copies of $\mathbb C P^3(-1,2,4,1)$ by the transformation~(\ref{5-4}).
Then, it is easy to show that the exceptional Laurent series~(ii) is converted to the regular series~(i)
in the $(\widetilde{x}, \widetilde{y}, \widetilde{z})$-chart,
and it approaches to the f\/ixed point $(\widetilde{Y}_1, \widetilde{Z}_1, \widetilde{\varepsilon }_1) = (1/2,0,0)$
as $z\to z_0$.

To construct the space of initial conditions of $(\text{P}_{\text{III} (D_8)})$,
we perform the weighted blow-ups at the points $(Y_1, Z_1, \varepsilon _1) = (1/2,0,0)$ and
$(\widetilde{Y}_1, \widetilde{Z}_1, \widetilde{\varepsilon }_1) = (1/2,0,0)$.

(i) blow-up at $(Y_1, Z_1, \varepsilon _1) = (1/2,0,0)$.

For the vector f\/ield (\ref{5-2}), put $\hat{Y}_1 = Y_1-1/2$ and
\begin{gather*}
u = \hat{Y}_1 - 2Z_1, \qquad v = Z_1,\qquad w = \varepsilon _2.
\end{gather*}
Then, the linear part is diagonalized around $(-1,0,0)$.
To resolve the singularity $(u,v,w) = (0,0,0)$ of the resultant equation, we introduce the weighted blow-up def\/ined by
\begin{gather*}
u   = u_1^2   = v_2^2 u_2   = w_3^2 u_3, \qquad
v   = u_1^4v_1   = v_2^4   = w_3^4 v_3, \qquad
w   = u_1w_1   = v_2 w_2   = w_3.
\end{gather*}
The relation between the original coordinates $(x,y,z)$ and the new coordinates $(u_3, v_3, w_3)$ is given by
\begin{alignat*}{4}
& x = w_3, \qquad &&
  y = u_3 + 2v_3w_3^2 + \frac{1}{2w_3^2}, \qquad &&
z = v_3,& \\
&  u_3 = y-2x^2z - \frac{1}{2x^2}, \qquad &&
w_3 = x, \qquad &&
v_3 = z.&
\end{alignat*}
In the new coordinates, $(\text{P}_{\text{III} (D_8)})$ is transformed into the Hamiltonian system
with the Hamiltonian function
\begin{gather*}
zH_1 = \frac{u}{2} + u^2w^2 + w^2z - \frac{2}{3}w^3z + 4uw^4z + 4w^6z^2 - \frac{w^2z}{1+2uw^2+4w^4z}
\end{gather*}
(here, the subscript is omitted for simplicity).
Furthermore, we can verify the symplectic relation~(\ref{3-15}).

The blow-up at $(\widetilde{Y}_1, \widetilde{Z}_1, \widetilde{\varepsilon }_1) = (-1,0,0)$
is done in the same way and the Hamiltonian function is easily obtained by applying the transformation (\ref{5-4})
to the above~$H_1$.
In this manner, we can obtain the space of initial conditions.

\section{The f\/ifth Painlev\'{e} equation}\label{section6}

The orbifold $\mathbb C P^3(1,0,1,1)$ for $(\text{P}_{\text{V}})$ is covered by three inhomogeneous coordinates
$(Y_1, Z_1, \varepsilon _1)$, $(X_3, Y_3, \varepsilon _3)$ and $(x,y,z)$ related as
\begin{gather*}
x =   \varepsilon _1^{-1} =   X_3\varepsilon _3^{-1},\qquad
y =   Y_1 =   Y_3, \qquad
z =   Z_1\varepsilon_1 ^{-1} =   \varepsilon _3^{-1}.
\end{gather*}
The second chart does not appear because $q = 0$.
We give the f\/ifth Painlev\'{e} equation on the local chart~$(x,y,z)$.
On the other local charts, $(\text{P}_{\text{V}})$ is expressed as rational dif\/ferential equations.
By rewriting them as $3$-dimensional autonomous vector f\/ields, we obtain
\begin{gather}
  \dot{Y}_1 = -2Y_1 + 2Y_1^2 - Y_1Z_1 + Y_1^2Z_1
 +\alpha_1 \varepsilon _1 - (\alpha _1 + \alpha _3) Y_1\varepsilon _1,\nonumber\\
 \dot{Z}_1 =-Z_1 + 2Y_1Z_1-Z_1^2+\varepsilon _1Z_1+2Y_1Z_1^2
 -(\alpha _1+\alpha _3)Z_1\varepsilon _1 + \alpha _2\varepsilon _1 Z_1^2, \label{6-2}\\
  \dot{\varepsilon }_1 = \varepsilon _1 (-1 + 2Y_1 - Z_1 + 2Y_1Z_1
 -(\alpha _1 + \alpha _3)\varepsilon _1 + \alpha _2Z_1\varepsilon _1),
\nonumber\\
 \dot{X}_3 = X_3 + X_3^2-2X_3Y_3 - 2X_3^2Y_3
 - \alpha _2 \varepsilon _3 + (\alpha _1 + \alpha _3-1) \varepsilon _3 X_3, \nonumber\\
  \dot{Y}_3 = -Y_3+Y_3^2-2X_3Y_3 + 2X_3Y_3^2
 + \alpha _1\varepsilon _3 - (\alpha _1 + \alpha _3) \varepsilon _3Y_3 ,\label{6-3}\\
 \dot{\varepsilon }_3 = -\varepsilon _3^2.\nonumber
\end{gather}
The autonomous limit is a Hamiltonian system obtained by putting $\varepsilon _3 = 0$ for equation~(\ref{6-3}).
On the set $\{ \varepsilon _3 = 0\}$, the action of the extended af\/f\/ine Weyl group
$\widetilde{W}(A_3^{(1)}) \simeq \operatorname{Aut} (A_3^{(1)}) \ltimes W(A_3^{(1)})$
is reduced to the action of $\operatorname{Aut} (A_3^{(1)})$ generated by
$(X_3, Y_3) \mapsto (Y_3-1, -X_3)$ and $(X_3, Y_3) \mapsto (X_3, 1-Y_3)$.

The vector f\/ield (\ref{6-2}) has f\/ixed points on the inf\/inity set
$\mathbb C P^2(1,0,1)$ given by
\begin{gather*}
(Y_1, Z_1, \varepsilon _1) = (0,0,0),\  (1,0,0),\  (0,-1,0),\  (1/2,-2,0),\  (1,-1,0).
\end{gather*}
The Jacobi matrices of (\ref{6-2}) at these f\/ixed points are
\begin{gather*}
-\left(
\begin{matrix}
2 & 0 & -\alpha_1 \\
0 & 1 & 0 \\
0 & 0 & 1
\end{matrix}
\right), \
\left(
\begin{matrix}
2 & 0 & -\alpha_3 \\
0 & 1 & 0 \\
0 & 0 & 1
\end{matrix}
\right), \
\left(
\begin{matrix}
-1 & 0 & \alpha_1 \\
0 & 1 & * \\
0 & 0 & 0
\end{matrix}
\right), \
\left(
\begin{matrix}
0 & -1/4 & * \\
4 & 0 & * \\
0 & 0 & 0
\end{matrix}
\right), \
\left(
\begin{matrix}
1 & 0 & -\alpha_3 \\
0 & -1 & * \\
0 & 0 & 0
\end{matrix}
\right),
\end{gather*}
respectively, where~$*$ denotes certain long numerical expressions.
The latter three f\/ixed points having zero eigenvalues correspond to the irregular singular point.
The center-(un)stable manifolds at these points can be exactly calculated for certain specif\/ic values of parameters,
which give the Riccati equations of conf\/luent hypergeometric type (Table~\ref{table3}).
The former two points correspond to movable singularities associated with
the regular Laurent series solutions~(ii) and~(i) given in Section~\ref{section2.2}.
The Laurent series~(ii) and~(i) converge to the points $(Y_1, Z_1, \varepsilon _1) = (0,0,0)$
and $(1,0,0)$ as $z\to z_0$, respectively.
The exceptional Laurent series solutions~(iii) and~(iv) do not converge to some point on
$\mathbb C P^3(1,0,1,1)$ as $z \to z_0$.

To treat the Laurent series~(iii) and~(iv), we use the B\"{a}cklund transformation~$\pi$ def\/ined by
\begin{gather*}
(\widetilde{x}, \widetilde{y}, \widetilde{z}, \widetilde{\alpha }_1, \widetilde{\alpha }_2, \widetilde{\alpha }_3)
 =  \pi (x,y,z,\alpha_1, \alpha _2, \alpha _3 )
 =  \left( z(y-1),  -\frac{x}{z},  z,  \alpha _2,  \alpha _3,  1-\alpha _1-\alpha _2-\alpha _3 \right).
\end{gather*}
We consider another space $\mathbb C P^3(1,0,1,1)$ with inhomogeneous coordinates denoted by
$(\widetilde{x}, \widetilde{y}, \widetilde{z})$ etc.
We glue two copies of $\mathbb C P^3(1,0,1,1)$ by the transformation $\pi$.
Then, the exceptional Laurent series (iii) and (iv) are converted to the regular series (i) and (ii), respectively,
in the $(\widetilde{x}, \widetilde{y}, \widetilde{z})$-chart.
They converge to the f\/ixed points $(\widetilde{Y}_1, \widetilde{Z}_1, \widetilde{\varepsilon }_1) = (1,0,0)$
and $(0,0,0)$ as $z\to z_0$, respectively.

To construct the space of initial conditions of $(\text{P}_{\text{V}})$,
we perform the weighted blow-ups at these four points.

(i) blow-up at $(Y_1, Z_1, \varepsilon _1) = (0,0,0)$.

\looseness=-1
For the vector f\/ield (\ref{6-2}), put $u = Y_1 - \alpha_1 \varepsilon _1$, $v = Z_1$, $w = \varepsilon _1$.
Then, the linear part is diagonalized.
For the system of $(u,v,w)$, we introduce the weighted blow-up def\/ined by
\begin{gather*}
u   = u_1^2   = v_2^2 u_2   = w_3^2 u_3, \qquad
v   = u_1v_1   = v_2   = w_3 v_3, \qquad
w   = u_1w_1   = v_2 w_2   = w_3.
\end{gather*}
The weight $(2,1,1)$ is taken from the eigenvalues of the Jacobi matrix at the singularity.
The relation between the original coordinates $(x,y,z)$ and the new coordinates $(u_3, v_3, w_3)$ is given by
\begin{alignat*}{4}
&
x = 1/w_3, \qquad &&
y = u_3w_3^2 + \alpha _1w_3, \qquad &&
z = v_3,&\\
& u_3 = x^2y - \alpha _1x, \qquad &&
w_3 = 1/x, \qquad &&
v_3 = z.&
\end{alignat*}
In the new coordinates, $(\text{P}_{\text{V}})$ is transformed into the Hamiltonian system, whose
Hamiltonian function is
\begin{gather*}
zH_1 = u -u^2w^2 - u^2w^3 z
 + (z-\alpha _1 + \alpha _3)uw - (2\alpha _1 + \alpha _2) uw^2z - \alpha _1(\alpha _1 + \alpha _2)wz
\end{gather*}
(here, the subscript is omitted for simplicity).
Furthermore, we can verify the symplectic relation~(\ref{3-15}).

(ii) blow-up at $(Y_1, Z_1, \varepsilon _1) = (1,0,0)$.

For the vector f\/ield (\ref{6-2}), put $\hat{Y}_1 = Y_1-1$
and $u = \hat{Y}_1 - \alpha_3 \varepsilon _1$, $v = Z_1$, $w = \varepsilon _1$.
Then, the linear part is diagonalized around $(1,0,0)$.
For the system of $(u,v,w)$, we introduce the weighted blow-up def\/ined by
\begin{gather*}
u   = u_4^2   = v_5^2 u_5   = w_6^2 u_6, \qquad
v   = u_4v_4   = v_5   = w_6 v_6, \qquad
w   = u_4w_4   = v_5 w_5   = w_6.
\end{gather*}
The \looseness=-1 relation between the original coordinates $(x,y,z)$ and the coordinates $(u_6, v_6, w_6)$ is given by
\begin{alignat*}{4}
& x = 1/w_6, \qquad &&
y = u_6w_6^2 + \alpha _3 w_6 + 1, \qquad &&
z = v_6,&\\
& u_6 = x^2y - x^2 - \alpha _3x, \qquad &&
w_6 = 1/x, \qquad &&
v_6 = z.&
\end{alignat*}
In the new coordinates, $(\text{P}_{\text{V}})$ is transformed into the Hamiltonian system, whose
Hamiltonian function is
\begin{gather*}
zH_1 = -u -u^2w^2 - u^2w^3 z
 + (-z + \alpha _1 - \alpha _3)uw - (2\alpha _3 + \alpha _2) uw^2z - \alpha _3(\alpha _2 + \alpha _3)wz
\end{gather*}
(here, the subscript is omitted for simplicity).
Again, we can verify the symplectic relation (\ref{3-15}).

The singularities $(\widetilde{Y}_1, \widetilde{Z}_1, \widetilde{\varepsilon }_1) = (1,0,0)$ and $(0,0,0)$
are resolved by the same way as above by the weighted blow-up with the weight $(2,1,1)$.
The result is easily obtained by the B\"{a}cklund transformation $\pi $ as in the previous sections.
In this manner, all singularities are resolved and it turns out that
the space of initial conditions of $(\text{P}_{\text{V}})$ is given by
f\/ive copies of $\mathbb C^2$ glued by the symplectic transformations.

\section{The sixth Painlev\'{e} equation}\label{section7}

The orbifold $\mathbb C P^3(1,0,0,1)$ for $(\text{P}_{\text{VI}})$ is covered by two inhomogeneous coordinates
$(Y_1, Z_1, \varepsilon _1)$ and $(x,y,z)$ related as
\begin{gather*}
x =   \varepsilon _1^{-1}, \qquad
y =   Y_1,\qquad
z =   Z_1.
\end{gather*}
We give the sixth Painlev\'{e} equation on the local chart $(x,y,z)$.
On the other local chart $(Y_1, Z_1, \varepsilon _1)$,
$(\text{P}_{\text{VI}})$ is expressed as a rational dif\/ferential equation.
As before, we rewrite it as a~$3$-dimensional autonomous polynomial vector f\/ield, whose
expression is too long and omitted here.
This vector f\/ield has f\/ixed points
\begin{gather*}
(Y_1, Z_1, \varepsilon _1) = (0,c,0),\   (1,c,0),\  (c,c,0),
\end{gather*}
where $c\in \mathbb C$ is an arbitrary constant.
The Jacobi matrices at these f\/ixed points are
\begin{gather*}
c \left(
\begin{matrix}
2 & 0 & -\alpha_4 \\
0 & 0 & c-1 \\
0 & 0 & 1
\end{matrix}
\right), \quad
(1-c) \left(
\begin{matrix}
2 & 0 & -\alpha_3 \\
0 & 0 & -c \\
0 & 0 & 1
\end{matrix}
\right), \quad
c(c-1) \left(
\begin{matrix}
2 & -2 & 1-\alpha _0 \\
0 & 0 & 1 \\
0 & 0 & 1
\end{matrix}
\right),
\end{gather*}
respectively.
These f\/ixed points correspond to movable singularities associated with
the regular Laurent series solutions (iii), (ii) and (i) given in Section~\ref{section2.2}.
The Laurent series~(iii),~(ii) and~(i) converge to these points as $z\to z_0$, respectively.
The exceptional Laurent series solutions~(iv) and~(v) do not converge to some point on
$\mathbb C P^3(1,0,0,1)$ as $z \to z_0$.

To treat the Laurent series~(iv) and~(v), we use the B\"{a}cklund transformation~$\pi_1$ def\/ined by
\begin{gather*}
(\widetilde{x}, \widetilde{y}, \widetilde{z},\widetilde{\alpha }_0,
 \widetilde{\alpha }_1, \widetilde{\alpha }_2, \widetilde{\alpha }_3,\widetilde{\alpha }_4)
 =  \pi_1 (x,y,z,\alpha _0, \alpha_1, \alpha _2, \alpha _3, \alpha_4 )   \\
\hphantom{(\widetilde{x}, \widetilde{y}, \widetilde{z},\widetilde{\alpha }_0,
 \widetilde{\alpha }_1, \widetilde{\alpha }_2, \widetilde{\alpha }_3,\widetilde{\alpha }_4)}{}
=  \left( -\frac{y}{z}(xy + \alpha _2),  \frac{z}{y},  z,
 \alpha _3,  \alpha _4,  \alpha _2, \alpha _0, \alpha _1 \right).
\label{7-4}
\end{gather*}
We consider another space $\mathbb C P^3(1,0,0,1)$ with inhomogeneous coordinates denoted by
$(\widetilde{x}, \widetilde{y}, \widetilde{z})$ etc.
We glue two copies of $\mathbb C P^3(1,0,0,1)$ by the transformation~$\pi_1$.
Then, the exceptional Laurent series (iv) becomes the regular series~(iii)
in the $(\widetilde{x}, \widetilde{y}, \widetilde{z})$-chart.
It converges to the f\/ixed point $(\widetilde{Y}_1, \widetilde{Z}_1, \widetilde{\varepsilon }_1) = (0,c,0)$ as $z\to z_0$.
On the other hand, the exceptional Laurent series (v) is expressed as
$\widetilde{x}(z) \sim O(1)$, $\widetilde{y}(z) \sim O(T)$ in the $(\widetilde{x}, \widetilde{y}, \widetilde{z})$-chart.
Since the solution is holomorphic at $z = z_0$, we need not resolve the singularity.

To construct the space of initial conditions of $(\text{P}_{\text{VI}})$,
we perform the weighted blow-ups at four points.

(i) blow-up at $(Y_1, Z_1, \varepsilon _1) = (0,c,0)$.

For the vector f\/ield written in $(Y_1, Z_1, \varepsilon _1)$-chart,
put $u = Y_1 - \alpha_4 \varepsilon _1$, $v = Z_1$, $w = \varepsilon _1$.
Then, we introduce the weighted blow-up with the weight $(2,0,1)$ def\/ined by
\begin{gather*}
u   = u_1^2   = w_2^2 u_2, \qquad
v   = v_1   = v_2, \qquad
w   = u_1w_1   = w_2.
\end{gather*}
This is a blow-up along the line $\{ v = \text{const}\}$.
The relation between the original coordinates $(x,y,z)$ and the new coordinates $(u_2, v_2, w_2)$ is given by
\begin{alignat*}{4}
& x = 1/w_2, \qquad &&
y = u_2w_2^2 + \alpha _2w_2, \qquad &&
z = v_2,&\\
& u_2 = x^2y - \alpha _2x, \qquad &&
w_2 = 1/x, \qquad &&
v_2 = z.&
\end{alignat*}
In the new coordinates, $(\text{P}_{\text{VI}})$ is transformed into the Hamiltonian system, whose
Hamiltonian function is
\begin{gather*}
z(z-1)H_1  =  u^3 w^4 + u z - u^2 w^2 (1+z)
  + (1 - \alpha _0 - \alpha _3 + 2 \alpha _4)u^2 w^3 \\
  \hphantom{z(z-1)H_1  =}{}
  + (\alpha_0- \alpha_4 + z \alpha_3 - z \alpha_4 -1)uw
  + \alpha_4 \big(\alpha _1 \alpha _2 + \alpha _2^2 + (1 - \alpha _0 - \alpha_3) \alpha _4 \big) w   \\
  \hphantom{z(z-1)H_1  =}{}
  + \big(\alpha _1 \alpha _2+\alpha _2^2+2 \alpha _4-2 \alpha _0 \alpha_4-2 \alpha _3 \alpha _4+\alpha _4^2 \big) u w^2
\end{gather*}
(here, the subscript is omitted for simplicity).
Furthermore, we can verify the symplectic relation~(\ref{3-15}).

The other singularities $(Y_1, Z_1, \varepsilon _1) = (1,c,0), (c,c,0)$
and $(\widetilde{Y}_1, \widetilde{Z}_1, \widetilde{\varepsilon }_1) = (0,c,0)$
are resolved by the same way as above by the weighted blow-up with the weight $(2,0,1)$.
In this manner, all singularities are resolved and
the space of initial conditions of $(\text{P}_{\text{VI}})$ is obtained by
six copies of $\mathbb C^2$ glued by the symplectic transformations; i.e.,
$(x,y)$, $(\widetilde{x}, \widetilde{y})$ and four charts arising from
blow-ups at four points. We need $(\widetilde{x}, \widetilde{y})$-chart for the exceptional
Laurent series~(v).

\pdfbookmark[1]{References}{ref}
\LastPageEnding


\begin{thebibliography}{99}
\footnotesize\itemsep=0pt

\bibitem{Bou}
Boutroux P., Recherches sur les transcendantes de {M}. {P}ainlev\'e et
 l'\'etude asymptotique des \'equations dif\/f\'erentielles du second ordre,
 \textit{Ann. Sci. \'Ecole Norm. Sup.~(3)} \textbf{30} (1913), 255--375.



\bibitem{Chi2}
Chiba H., Kovalevskaya exponents and the space of initial conditions of a
 quasi-homogeneous vector f\/ield, \href{http://dx.doi.org/10.1016/j.jde.2015.08.035}{\textit{J.~Differential Equations}}
 \textbf{259} (2015), 7681--7716, \href{http://arxiv.org/abs/1407.1511}{arXiv:1407.1511}.

\bibitem{Chi1}
Chiba H., The f\/irst, second and fourth {P}ainlev\'e equations on weighted
 projective spaces, \href{http://dx.doi.org/10.1016/j.jde.2015.09.020}{\textit{J.~Differential Equations}} \textbf{260} (2016),
 1263--1313, \href{http://arxiv.org/abs/1311.1877}{arXiv:1311.1877}.

\bibitem{Chi3}
Chiba H., A compactif\/ied Riccati equation of Airy type on a weighted projective
 space, submitted.

\bibitem{Cho}
Chow S.N., Li C.Z., Wang D., Normal forms and bifurcation of planar vector
 f\/ields, \href{http://dx.doi.org/10.1017/CBO9780511665639}{Cambridge University Press}, Cambridge, 1994.

\bibitem{CosCos}
Costin O., Costin R.D., Singular normal form for the {P}ainlev\'e equation
 {P{$1$}}, \href{http://dx.doi.org/10.1088/0951-7715/11/5/002}{\textit{Nonlinearity}} \textbf{11} (1998), 1195--1208,
 \href{http://arxiv.org/abs/math.CA/9710209}{math.CA/9710209}.

\bibitem{Ohy}
Ohyama Y., Kawamuko H., Sakai H., Okamoto K., Studies on the {P}ainlev\'e
 equations. {V}.~{T}hird {P}ainlev\'e equations of special type {$P_{\rm
 III}(D_7)$} and {$P_{\rm III}(D_8)$}, \textit{J.~Math. Sci. Univ. Tokyo}
 \textbf{13} (2006), 145--204.

\bibitem{Sak}
Sakai H., Rational surfaces associated with af\/f\/ine root systems and geometry of
 the {P}ainlev\'e equations, \href{http://dx.doi.org/10.1007/s002200100446}{\textit{Comm. Math. Phys.}} \textbf{220} (2001),
 165--229.

\bibitem{Tsu}
Tsuda T., Okamoto K., Sakai H., Folding transformations of the {P}ainlev\'e
 equations, \href{http://dx.doi.org/10.1007/s00208-004-0600-8}{\textit{Math. Ann.}} \textbf{331} (2005), 713--738.

\end{thebibliography}
\end{document}